\documentclass[11pt
]{article}
\usepackage[letterpaper,margin=1in,footskip=0.25in]{geometry}
\usepackage{microtype}
\usepackage{cgstix}
\usepackage{mathtools, amssymb, amsthm}
\PassOptionsToPackage{dvipsnames, hyperref}{xcolor}
\usepackage{tikz-cd, extarrows}
\usetikzlibrary{decorations.pathmorphing}

\usepackage[shortlabels]{enumitem}
\setlist{nosep}
\setlist[enumerate]{label=\((\roman*)\),ref=(\roman*)}
\usepackage[style=alphabetic,url,doi,sortcites]{biblatex}
\usepackage{hyperref}
\usepackage[nameinlink]{cleveref}
\hypersetup{
  colorlinks=true,
  linkcolor={blue!50!black},
  citecolor={green!50!black},
  urlcolor={magenta!80!black},
  bookmarks=true,
  bookmarksdepth=2,
  pdftitle = {Correspondences in log Hodge cohomology},
  pdfauthor = {Charles Godfrey}
}
\usepackage[nomargin,marginclue,inline]{fixme}

\usepackage{graphicx}
\usepackage{rotating}
\usepackage{fancyhdr}

\newcommand{\field}[1]{\mathbb #1}

\newcommand{\QQ}{\field Q}
\newcommand{\NN}{\field N}

\newcommand{\sHom}{\mathcal{H}om}

\newcommand{\sExt}{\mathcal E\!xt}

\newcommand{\strshf}[1]{\mathscr{O}_{#1}}
\DeclareMathOperator{\Spec}{Spec}

\DeclareMathOperator{\pr}{pr}

\DeclareMathOperator{\colim}{colim}

\DeclareMathOperator{\tr}{tr}

\DeclareMathOperator{\Sing}{Sing}

\DeclareMathOperator{\supp}{supp}

\DeclareMathOperator{\cl}{cl}

\newcommand{\Ab}{\mathbf{Ab}}

\newcommand{\inj}{\hookrightarrow}
\newcommand{\isom}{\xrightarrow{\sim}}

\newcommand{\cK}{\mathcal{K}}

\newcommand{\sG}{\mathscr{G}}

\newcommand{\sI}{\mathscr{I}}

\newcommand{\sO}{\mathscr{O}}

\DeclareMathOperator{\codim}{codim}
\DeclareMathOperator{\Bl}{Bl}
\DeclareMathOperator{\snc}{snc}

\DeclareMathOperator{\cor}{cor}

\newcommand{\schs}[1]{(#1, \Phi_{#1})}
\newcommand{\lsp}[1]{(#1, \Delta_{#1})}
\newcommand{\lsps}[1]{(#1, \Delta_{#1}, \Phi_{#1})}

\newcommand{\ologd}[2]{\Omega_{#1}^{#2}(\log \Delta_{#1})}

\newcommand{\lstrshf}[1]{\mathscr{O}_{#1}(-\Delta_{#1})}
\newcommand{\lcanshf}[1]{\omega_{#1}(\Delta_{#1})}

\numberwithin{equation}{section}

\theoremstyle{plain}
\newtheorem{theorem}[equation]{Theorem}
\crefname{theorem}{Theorem}{Theorems}
\newtheorem*{theorem*}{Theorem}
\newtheorem{lemma}[equation]{Lemma}
\crefname{lemma}{Lemma}{Lemmas}
\newtheorem{proposition}[equation]{Proposition}
\crefname{proposition}{Proposition}{Propositions}
\newtheorem{conjecture}[equation]{Conjecture}
\crefname{conjecture}{Conjecture}{Conjectures}
\newtheorem{corollary}[equation]{Corollary}
\crefname{corollary}{Corollary}{Corollaries}

\crefname{problem}{Problem}{Problems}

\theoremstyle{plain}

\newtheoremstyle{named}{}{}{\itshape}{}{\bfseries}{.}{.5em}{#1 \thmnote{#3}}
\theoremstyle{named}

\theoremstyle{definition}
\newtheorem{definition}[equation]{Definition}
\crefname{definition}{Definition}{Definitions}

\crefname{variant}{Variant}{Variants}

\crefname{notation}{Notation}{Notations}
\newtheorem{convention}[equation]{Convention}
\crefname{convention}{Convention}{Conventions}

\crefname{claim}{Claim}{Claims}

\crefname{slogan}{Slogan}{Slogans}

\crefname{fact}{Fact}{Facts}

\crefname{assumption}{Assumption}{Assumptions}

\crefname{hypothesis}{Hypothesis}{Hypotheses}

\crefname{construction}{Construction}{Constructions}

\crefname{calculation}{Calculation}{Calculations}

\theoremstyle{remark}
\newtheorem{remark}[equation]{Remark}
\crefname{remark}{Remark}{Remarks}

\crefname{observation}{Observation}{Observations}
\newtheorem{example}[equation]{Example}
\crefname{example}{Example}{Examples}
\newtheorem*{example*}{Example}

\crefname{question}{Question}{Questions}
\newtheorem*{question*}{Question}

\crefname{warning}{Warning}{Warnings}
\crefname{enumi}{}{}
\crefname{figure}{Figure}{Figure}

\title{Correspondences in log Hodge cohomology}
\author{Charles Godfrey\\
Pacific Northwest National Laboratory\\
\href{mailto:charles.godfrey@pnnl.gov}{\texttt{charles.godfrey@pnnl.gov}}\\
}
\date{\today} 

\begin{document}
\maketitle
\thispagestyle{fancy}
\renewcommand{\headrule}{}
\renewcommand{\footrulewidth}{0.5pt}
\fancyhead[L,C,R]{}
\fancyfoot[L]{\footnotesize This work was completed while the
author was a PhD student in the University of Washington Department of Mathematics. The author was partially supported by the University of
Washington Department of Mathematics Graduate Research Fellowship, and by the
NSF grant DMS-1440140, administered by the Mathematical Sciences  Research
Institute, while in residence at MSRI during the program Birational Geometry
and Moduli Spaces.}
\fancyfoot[C,R]{}
\abstract{%
We construct correspondences in logarithmic Hodge theory over a perfect field of
arbitrary characteristic. These are represented by classes in the cohomology of
sheaves of differential forms with log poles and, notably, log \emph{zeroes} on
cartesian products of varieties. 
From one perspective this generalizes work of Chatzistamatiou and R\"ulling,
who developed (non-logarithmic) Hodge correspondences over
perfect fields of arbitrary characteristic; from another we provide partial
generalizations of more recent work of Binda, Park and {\O}stv{\ae}r on logarithmic Hodge correspondences by relaxing finiteness and strictness
conditions on the correspondences considered.
}
 

\section{Introduction}
\label{sec:intro}

Generally speaking, a \emph{correspondence} between two algebraic varieties \(X
\) and \(Y \) over a field \(k\) is a cycle or cohomology class on the product
\(X \times Y \). The study of such objects dates back (at least) to Lefschetz
\cite{lefschetzAlgebraicGeometry1953}, and features prominently in famous
conjectures on algebraic cycles (see e.g. \cite{voisin_decomp_diag}) and
Voevodsky's theory of motives (see e.g. \cite{MR2242284}). 

In a number of algebro-geometric research areas it has become commonplace to
work with pairs \(\lsp{X}\) consisting of a variety \(X\) together with a
divisor \(\Delta_X \) on \(X\). Such areas include moduli of varieties (where
pairs generalize the curves with marked points of
\cite{deligneIrreducibilitySpaceCurves1969}), birational geometry (where pairs
appear naturally, for example as the output of strong resolution of
singularities \cite{MR1658959}) and logarithmic geometry (in this case vast
generalizations of divisors \(\Delta_X\) are allowed \cite{MR3838359}). It is
natural to wonder about analogues of correspondences in this category of pairs,
and there have been efforts in this direction, for example development of
categories of logarithmic motives
\cite{bindaTriangulatedCategoriesLogarithmic2020}. 

In this paper, we focus on correspondences for logarithmic Hodge cohomology of
pairs \(\lsp{X}\), where \(X\) is a smooth (but not necessarily proper)
variety over a perfect field \(k\) and \(\Delta_X\) is a simple normal crossing
divisor on \(X \). These cohomology groups can be described as
\begin{equation}
  \label{eq:log-hodge-teaser}
  H^\ast \lsp{X} = \bigoplus H^q(X, \Omega_X^p(\log \Delta_X)),
\end{equation}
where \(\Omega_X(\log \Delta_X)\) is the sheaf of differential \(1\)-forms on
\(X\) with log poles along \(\Delta_X\) and \(\Omega_X^p(\log \Delta_X) \) the
\(p\)-th exterior power thereof. In addition we consider a generalization
where \(X\) comes with a \emph{family of supports} \(\Phi_X\), and the ordinary
cohomology groups on the right hand side of \cref{eq:log-hodge-teaser} are
replaced with cohomology with supports in \(\Phi_X\), namely \(H^q_{\Phi_X}(X,
\Omega_X^p(\log \Delta_X))\). Allowing for supports greatly expands the
applicability of our results: for example, it permits us to construct a
correspondence associated to a cycle \(Z \subset X \times Y \) in a situation
where neither \(X \) nor \(Y \) is proper over \(k\), but \(Z \) is proper over
both \(X \) and \(Y \).\footnote{One way that such a cycle \(Z\) might naturally
arise is as the closure of the graph of a birational equivalence \(X
\dashrightarrow Y\) of non-proper varieties.}

There are multiple motivations for investigating
correspondences for this particular cohomology of pairs: 
\begin{itemize}
  \item By analogy with the case of varieties (that is, without auxiliary
  divisors/log structures), we suspect that correspondences at the level of Chow
  cycles are more fundamental, and that (many) correspondences in logarithmic
  Hodge cohomology are obtained from Chow correspondences via a cycle morphism.
  However, as of this writing there is no full-fledged theory of Chow cohomology
  of pairs or log schemes (though there has been considerable progress, for
  instance in \cite{bohning2022prelog,barrott2018logarithmic}). Logarithmic
  Hodge cohomology is in contrast quite mature, appearing as early as
  \cite{MR498551}.
  \item Correspondences in (non-logarithmic) Hodge cohomology have found remarkable applications.
  For example, \cite{MR2923726} used them to prove birational invariance of the
  cohomology groups of the structure sheaf \(H^i(X, \strshf{X})\) for smooth
  varieties \(X\) over perfect fields of positive characteristic. In fact, attempting to implement a similar
  strategy with \emph{logarithmic} Hodge cohomology to obtain results on
  invariance of the cohomology groups  \(H^i(X, \strshf{X}(-\Delta_X))\) with
  respect to (a restricted class of) birational equivalences was the initial
  inspiration for this work. Ultimately that attempt was unsuccessful, as we
  describe in \Cref{sec:attempts}.
  \item There has been recent interest in logarithmic Hodge cohomology as a
  representable functor on a category of motives of log schemes over a perfect
  field \cite[\S 9]{bindaTriangulatedCategoriesLogarithmic2020}. While that work
  does also construct some correspondences, they are restricted to those
  associated with logarithmic Hodge cohomology classes of cycles \(Z \subset X
  \times Y\) which are \emph{finite} over \(X\) and obey additional strictness
  (in the sense of logarithmic geometry) conditions; we remove these restrictions.
\end{itemize}

The correspondences we construct are obtained from certain Hodge classes with
both log poles \emph{and} log zeroes. Our main result is: 
\begin{theorem}[= \cref{thm:log-hodge-corresp}]
  \label{thm:log-hodge-corresp-teaser}
  A class \(\gamma \in H^j_{P(\Phi_X, \Phi_Y)}(X \times Y, \ologd{X\times
  Y}{i}(-\mathrm{pr}_X^* \Delta_X))\) defines homomorphisms  
  \[    \cor (\gamma) : H_{\Phi_X}^q(X, \ologd{X}{p}) \to
  H_{\Phi_Y}^{q+j-d_X}(Y, \ologd{Y}{p+i-d_X}) \] by the formula
  \(\cor(\gamma)(\alpha) := \mathrm{pr}_{Y*}(\mathrm{pr}_X^* (\alpha) \smile
  \gamma)\). Moreover if \(\lsps{Z}\) is another snc pair with supports and
  \(\delta \in H^{j'}_{P(\Phi_Y, \Phi_Z)}(Y \times Z, \ologd{Y \times
  Z}{i'}(-\mathrm{pr}_Y^* \Delta_Y))\), then \[\mathrm{pr}_{X \times Z
  *}(\mathrm{pr}_{X \times Y}^* (\gamma) \smile \mathrm{pr}_{Y \times Z}^*
  (\delta)) \in H_{P(\Phi_X,\Phi_Z)}^{j  + j' - d_Y}(X \times Z, \ologd{X \times
  Z}{i + i' - d_Y}(-\mathrm{pr}_X^* \Delta_X)) \text{  and}\]
  \[\cor(\mathrm{pr}_{X \times Z *}(\mathrm{pr}_{X \times Y}^* (\gamma) \smile
  \mathrm{pr}_{Y \times Z}^* (\delta)) ) = \cor(\delta)\circ\cor(\gamma)  \] as
  homomorphisms \(H_{\Phi_X}^q(X, \ologd{X}{p}) \to H_{\Phi_Z}^{q+ j + j' - d_X
  - d_Y}(Z, \ologd{Z}{p+i+i' - d_X - d_Y})\). 
\end{theorem}
In the above, \(\Delta_{XY} := \mathrm{pr}_X^* \Delta_X + \mathrm{pr}_Y^*
\Delta_Y\), a simple normal crossing divisor on \(X \times Y\). There is a simple heuristic explanation for the appearance of
differential forms in \(\ologd{X\times Y}{i}(-\mathrm{pr}_X^* \Delta_X)\):
working over the complex numbers, in the case where \(X\) and \(Y\) are both proper the class \(\cor(\gamma)(\alpha) :=
\mathrm{pr}_{Y*}(\mathrm{pr}_X^* (\alpha) \smile \gamma)\) can be computed
explicitly as an integral of the form 
\begin{equation}
  \label{eq:integrate-over-x}
  \int_X \alpha(x) \wedge \gamma(x, y),
\end{equation}
and this integral will only be finite when the log poles of \(\alpha\) along
\(\Delta_X\) are cancelled by complementary zeroes of the form \(\gamma(x, y)\)
along the preimage \(\mathrm{pr}_X^* \Delta_X\).

Our proof of \cref{thm:log-hodge-corresp-teaser} relies heavily on prior work on
both Hodge cohomology with supports \cite[\S 2]{MR2923726} and its logarithmic
variant \cite[\S 9]{bindaTriangulatedCategoriesLogarithmic2020}.
\Cref{sec:func-prop-log-hodge} is a rapid summary of those results. The key new
technical ingredient is a base change formula on the interaction of pushforward
and pullback operations in cartesian squares, proved in \Cref{sec:proj-form}.
\Cref{sec:correspondences} includes the proof of our main theorem. 

\subsection{Acknowledgements}
Thanks to Daniel Bragg, Yun Hao, Sarah Scherotzke, Nicol\`o Sibilla and
Mattia Talpo for helpful conversations, to Lawrence Jack Barrott for
illuminating email correspondence regarding logarithmic aspects of Chow and
Hodge, and to my advisor S\'andor Kov\'acs for many insightful discussions.
Thanks also to the participants of the Spring 2019 MSRI graduate student
seminar, in particular Giovanni Inchiostro and organizer Fatemeh Rezaee, for
feedback on early work on this paper.

\section{Functoriality properties of log Hodge cohomology with supports}
\label{sec:func-prop-log-hodge}

\subsection{Supports}
In order to obtain results that apply to correspondences between varieties \(X
\) and \( Y\) where neither  \(X \) nor \(Y\) is proper, it is necessary
to work with cohomology with \emph{ supports}, also known as local cohomology. A
primary source for the material of this subsection is \cite[\S IV]{MR0222093}. Let \(X\) be a noetherian scheme.

\begin{definition}[{\cite[\S IV]{MR0222093}, \cite[\S 1.1]{MR2923726}}] 
  \label{def:2}
  A \textbf{family of supports \(\Phi\) on \(X\)} is a non-empty
  collection \( \Phi \) of closed subsets of \(X\) such that
  \begin{itemize}
  \item If \( C \in \Phi \) and \( D \subset C \) is a closed
    subset, then \( D \in \Phi \).
  \item If \( C, D \in \Phi \) then \( C \cup D \in \Phi \).  
  \end{itemize}
\end{definition}

\begin{example}
  \label{ex:2}
  \( \Phi = \{ \text{  all closed subsets of  } X \, \,  \}  \) is a family of
  supports. More generally if \( \mathcal{C}\) is any collection of closed
  subsets  \(C \subset X\), there is a \emph{smallest} family of supports
  \(\Phi(\mathcal{C})\) containing \(\mathcal{C}\) (explicitly,
  \(\Phi(\mathcal{C})\) consists of finite unions \(\bigcup_{i} Z_{i} \) of
  closed subsets \(Z_{i} \subset C_{i}\) of elements \( C_{i} \in
  \mathcal{C}\)). Taking \(\Phi = \Phi(\{X\})\) recovers the previous example. A
  more interesting example is the case where for some fixed \(p \in \NN\), \(\Phi = \{ \text{closed sets } Z \subseteq X \, | \, \dim Z \leq p\}\).
\end{example}

There is a close relationship between families of supports on X and certain
collections of specialization-closed subsets of points on \(X \), and we can
also consider sheaves of families of supports --- for further details we refer
to \cite[\S IV.1]{MR0222093}.

If \( f : X \to Y \) is a morphism of noetherian schemes and
\( \Psi \) is a family of supports on \(Y\), then
\( \{ f^{-1}(Z) \, | \, Z \in \Psi \} \) is a family of closed subsets
of \(X\), and is closed under unions, but is \emph{not} in general
closed under taking closed subsets.
\begin{definition}
  \label{def:inv-image-supp}
  \( f^{-1} (\Psi) \) is the smallest family of supports on \(X\)
  containing \( \{ f^{-1}(Z) \, | \, Z \in \Psi \} \).
\end{definition}
 
Let \( \Phi \) be a family of supports on \(X\). The notation/terminology
\textbf{\( f|_\Phi \) is proper} will mean  \( f|_C \) is proper for every \( C
\in \Phi \). If \(f|_{\Phi}\) is proper then \( f(C) \subset Y\) is closed for
every \( C \in \Phi \) and in fact
\begin{equation}
  \label{eq:49}
   f(\Phi) = \{ f(C) \subset Y \, \, | \, \, C \in \Phi \}
\end{equation}
is a family of supports on \(Y\). The key point here is that if
\( D \subset f(C) \) is closed, then \( f^{-1}(D) \cap C \in \Phi \)
and \( D = f( f^{-1}(D) \cap C ) \).

\begin{definition}
  \label{def:sch-supp}
  A \textbf{scheme with supports \((X, \Phi_{X})\)} is a scheme \(X\)
  together with a family of supports \(\Phi_{X}\) on \(X\).  
\end{definition}


\begin{definition}
  \label{def:6}
  A \textbf{pushing morphism \(f: (X, \Phi_{X}) \to (Y, \Phi_{Y})\)} of schemes
  with supports is a morphism \(f: X \to Y\) of underlying schemes such that
  \(f|_{\Phi_{X}}\) is proper and \(f(\Phi_{X}) \subset \Phi_{Y}\).  A
  \textbf{pulling morphism \(f: X \to Y\)} is a morphism \(f: X \to Y\) such
  that \(f^{-1}(\Phi_{Y}) \subset \Phi_{X}\). 
\end{definition}

These morphisms provide two different categories  with underlying set of objects
schemes with supports \(\schs{X}\), and pushing/pulling morphisms respectively
(the verification is elementary; for instance a composition of pushing morphisms
is again a pushing morphism since compositions of proper morphisms are proper).
Schemes with supports provide a natural setting for describing functoriality
properties of local cohomology. Let \(\mathscr{F}\) be a
sheaf of abelian groups on a scheme with supports \((X,
\Phi_{X})\).\footnote{Simply put \(\mathscr{F}\) is a sheaf of abelian groups on \(X\).}  

\begin{definition}
  \label{def:7}
  The \textbf{sheaf of sections with supports} of \(\mathscr{F}\),
  denoted \(\underline{\Gamma}_{\Phi} (\mathscr{F})\), is obtained by setting
  \begin{equation}
    \label{eq:50}
    \begin{split}
      \underline{\Gamma}_{\Phi} (\mathscr{F})(U) & = \{\sigma \in \mathscr{F}(U) \, | \, \supp \sigma
      \in \Phi_{X} \rvert_{U} \, \}  \\
      & \\
    \end{split}
  \end{equation}
  for each open \( U \subset X \) (here \(\Phi_{X} \lvert_{U}\)
  is short for \(\iota^{-1} \Phi_{X}\) where \(\iota: U \to X\) is the
  inclusion). More explicitly: for a local section
  \(\sigma \in \mathscr{F}(U)\),
  \(\sigma \in \underline{\Gamma}_{\Phi} (\mathscr{F})(U) \) means
  \(\supp \sigma = C \cap U\) for a closed set \(C \subset \Phi_{X}\).
\end{definition}
The functor \(\underline{\Gamma}_{\Phi}\) is right adjoint to an exact functor,
for instance the inclusion of the subcategory \(\Ab_{\Phi}(X) \subset \Ab(X)\)
of abelian sheaves on \(X\) with supports in \(\Phi\); so,
\(\underline{\Gamma}_{\Phi}\) is left exact and preserves injectives. In the
case \(\Phi = \Phi(Z) \) for some closed \(Z \subset X\), this is proved in
\cite[\href{https://stacks.math.columbia.edu/tag/0A39}{Tag 0A39},
\href{https://stacks.math.columbia.edu/tag/0G6Y}{Tag
0G6Y},\href{https://stacks.math.columbia.edu/tag/0G7F}{Tag
0G7F}]{stacks-project}  --- the general case can then be obtained by writing
\(\underline{\Gamma}_{\Phi}\) as a filtered colimit:
\[\underline{\Gamma}_{\Phi} = \colim_{Z \in \Phi} \underline{\Gamma}_Z.\] The
right derived functor of \(\underline{\Gamma}_{\Phi}\) will be denoted
\(R\underline{\Gamma}_{\Phi} \). Taking global sections on \(X\) gives the
\textbf{sections with supports} of \(\mathscr{F}\): \(\Gamma_{\Phi}
(\mathscr{F}) := \Gamma_{X}(\underline{\Gamma}_{\Phi}( \mathscr{F}))\)
This is also left exact, and (the cohomologies of) its derived functor
give the \textbf{cohomology with supports in \(\Phi\)}: \(H_{\Phi}^{i}(X, \mathscr{F}):= R^{i} \Gamma_{\Phi}(\mathscr{F})\).

\begin{proposition}[
  ]   
  \label{prp:coho-with-supp}
  Cohomology with supports enjoys the following functoriality
  properties:
  \begin{enumerate}
  \item  \label{item:coho-with-supp1}If \(f: (X, \Phi_{X}) \to (Y, \Phi_{Y})\) is a \emph{pulling}
    morphism of schemes with supports, \(\mathscr{F}, \mathscr{G}\) are sheaves of abelian
    groups on \(X, Y\) respectively, and if
    \begin{equation}
      \label{eq:55}
      \varphi : \mathscr{G} \to f_{*} \mathscr{F} \text{ is a morphism
        of sheaves,}
    \end{equation}
    then there is a natural morphism
    \( R \underline{\Gamma}_{\Phi} \mathscr{G} \to R f_{*} R
    \underline{\Gamma}_{\Phi}\mathscr{F}\).  Similarly if
    \(\mathscr{F}\) and \(\mathscr{G}\) are quasicoherent then there
    are natural morphisms
    \( R \underline{\Gamma}_{\Phi} \mathscr{G} \to R f_{*} R
    \underline{\Gamma}_{\Phi}\mathscr{F}\).
  \item \label{item:coho-with-supp2} If \(f: (X, \Phi_{X}) \to (Y, \Phi_{Y})\) is a
    \emph{pushing} morphism, \(\mathscr{F}, \mathscr{G}\) are sheaves of
    abelian groups on \(X, Y\) respectively, and
    \begin{equation}
      \label{eq:198}
      \psi: Rf_{*} \mathscr{F} \to \mathscr{G} \text{ is a morphism in
        the derived category of } X,
    \end{equation}
    then there is a natural morphism
    \( R f_{*} R \underline{\Gamma}_{\Phi} (\mathscr{F}) \to R
    \underline{\Gamma}_{\Phi}\mathscr{G}\).
  \end{enumerate}
\end{proposition}

Both parts of the proposition follow from
\cite[\href{https://stacks.math.columbia.edu/tag/0G78}{Tag
0G78}]{stacks-project};  \Cref{item:coho-with-supp1} is discussed in detail in
\cite[\S 2.1]{MR2923726} and \cref{item:coho-with-supp2} can be extracted from
\cite[\S 2.2]{MR2923726} (although it doesn't appear to be stated explicitly).
See also \cite[Constructions 9.4.2,
9.5.3]{bindaTriangulatedCategoriesLogarithmic2020}

\subsection{Differential forms with log poles}
Let \(k\) be a perfect field. 

\begin{definition}
  \label{def:log-smooth-pairs}
  A \textbf{snc pair with supports \(\lsps{X}\)} over \(k\) is a smooth scheme
  \(X\) separated and of finite type over \(k\) with a family of supports
  \(\Phi_{X}\) together with a reduced, effective divisor \(\Delta_{X}\) on
  \(X\) such that \(\supp \Delta_{X}\) has simple normal crossings, in the sense
  that for any point \(x \in X\) there are regular parameters \(z_1 , \dots, z_c
  \in \sO_{X,x}\) such that \(\supp \Delta_X = V(z_1 \cdot z_2 \cdots z_r)\) on
  a Zariski neighborhood of \(x\).\footnote{This is equivalent to the more
  general definition \cite[Def.
  7.2.1]{bindaTriangulatedCategoriesLogarithmic2020} in the case where the base
  scheme is \(\Spec k\), which is all we need.} The \textbf{interior \(U_{X}\)} of
  a snc pair with supports \((X, \Delta_{X}, \Phi_{X})\) is
  \begin{equation}
    \label{eq:57}
    U_{X}: = X \setminus \supp \Delta_{X}
  \end{equation}
  The inclusion of \(U_{X}\) in \(X\) is denoted by
  \( \iota_{X}: U_{X} \to X \).
\end{definition}

Here \(\supp \Delta_{X}\)
denotes the \textbf{support} of \(\Delta_{X}\) (if \(\Delta_{X} = \sum_{i} a_{i}
D_{i}\) where the \(D_{i}\) are prime divisors, then \(\supp \Delta_{X} =
\cup_{i} D_{i}\)). Similarly let \(j_{X}: \supp \Delta_{X} \to X\) denote the
evident inclusion.

\begin{definition}[{compare with \cite[Def. 1.1.4]{MR2923726}}]
  \label{def:push-pull}
  A \textbf{pulling morphism \(f: \lsps{X} \to \lsps{Y}\) of
    snc pairs with supports} is a pulling morphism
  \(f: X \to Y\) of underlying schemes with support such that
    \(f^{-1}(\supp \Delta_{Y}) \subset \supp \Delta_{X}\); equivalently, \(f\)
    restricts to a morphism \(f|_{U_X}: U_X \to U_Y\).
  A \textbf{pushing morphism \(f: \lsps{X} \to \lsps{Y}\) of
    snc pairs with supports} is a pushing morphism of
  underlying schemes with support such that \(f^{*} \Delta_{Y} = \Delta_{X}\).   
\end{definition}
Note that if \(f: \lsps{X} \to \lsps{Y}\) is a pushing morphism then \(U_{X} =
f^{-1} (U_{Y})\), so for example if \(f:X \to Y\) is proper then so is the
induced map \(U_{X} \to U_Y\). 
\begin{convention}[{compare with \cite[1.1.5]{MR2923726}}]
  \label{conv:props-of-mps-of-pairs}
  A morphism of snc pairs with supports \(f: (X, \Delta_{X}, \Phi_{X}) \to
  \lsps{Y}\) is flat, proper, an immersion, etc. if and only if the same is true
  of the underlying morphism of schemes \(f: X\to Y\).  A diagram
  of snc pairs with supports
  \begin{equation}
    \label{eq:59}
    \begin{tikzcd}
      (X', \Delta_{X'}, \Phi_{X'}) \arrow[r, "g'"] \arrow[d, "f'"'] & (X, \Delta_{X},
      \Phi_{X}) \arrow[d, "f"] \\
      (Y', \Delta_{Y'}, \Phi_{Y'}) \arrow[r, "g"'] & \lsps{Y}
    \end{tikzcd}
  \end{equation}
  is \textbf{cartesian} if and only if the induced diagram of
  underlying schemes 
  \begin{equation}
    \label{eq:62}
    \begin{tikzcd}
      X' \arrow[r, "g'"] \arrow[d, "f'"'] \arrow[dr, phantom, "\Box"] &
      X \arrow[d, "f"] \\
      Y' \arrow[r, "g"'] & Y
    \end{tikzcd}
  \end{equation}
  is cartesian.\footnote{If we take the red pill of logarithmic
  geometry, it starts to seem almost more reasonable to only require flatness,
  properness, cartesianness and so on of the induced maps of \emph{interiors}
  \(U_X \to U_Y\). However we do use the stronger restrictions of the given
  definition in some of the proofs below.}
\end{convention}
The terminology is meant to suggest that pushing (resp. pulling) morphisms
induce pushforward (resp. pullback) maps on log Hodge cohomology, as we now
describe. 

If \(\lsp{X}\) is an snc pair, or more generally a normal separated scheme of
finite type \(X\) over \(k\) together with a sequence of effective Cartier
divisors \(D_1, \dots, D_N \subseteq X\) with sum \(\Delta_X = \sum_i D_i\),
then it comes with a sheaf of \emph{differential forms with log poles}
\(\Omega_X(\log \Delta_X)\). In the case where \(\lsps{X}\) is snc, this sheaf
and its properties are described in \cite[\S 2]{MR1193913}. For a definition and
treatment of \(\Omega_X(\log \Delta_X)\) in the much greater generality of
logarithmic schemes we refer to \cite[\S IV]{MR3838359}. 

In some of the calculations below the following concrete local description will
be very useful. Let \(z_{1}, z_{2}, \dots, z_{n}\) be local coordinates at a
point \(x \in X\) such that \(\supp \Delta_{X} = V(z_{1} z_{2} \cdots z_{r})\)
in a neighborhood of \(x\). Recall that as \(X\) is smooth the differentials \(d
\, z_{1}, d \, z_{2}, \dots, d \, z_{n}\) freely generate \(\Omega_{X}\) on a
neighborhood of \(x\). 
\begin{lemma}[{see e.g. \cite[\S 2]{MR1193913}}]
  \label{lem:loc-descr-diff-log-poles}
  The sections
  \( \frac{d \, z_{1}}{z_{1}}, \dots, \frac{d \, z_{r}}{z_{r}}, d \,
    z_{r+1}, \dots, d \, z_{n} \) freely generate
  \(\Omega_{X}(\log \Delta_{X})\) on a neighborhood of \(x\).
\end{lemma}
Given \(\Omega_X(\log \Delta_X)\), we can form
the exterior powers 
\begin{equation}
  \label{eq:ext-pow-om}
  \Omega_X^p(\log \Delta_X) := \bigwedge^p  \Omega_X(\log \Delta_X),
\end{equation}
and combining \cref{lem:loc-descr-diff-log-poles} with \eqref{eq:ext-pow-om}
gives concrete local descriptions of the \(\Omega_X^p(\log \Delta_X)\); in
particular, we see that \(\Omega_X^{\dim X}(\log \Delta_X) = \omega_X(\Delta_X)\).
\begin{definition}
  \label{def:1}
  The \textbf{log-Hodge cohomology with supports} of a log-smooth pair with
  supports \(\lsps{X}\) is defined by
  \begin{equation}
    \label{eq:99}
    H^{d} \lsps{X} = \bigoplus_{p+q = d} H_{\Phi}^{q}(X,
    \Omega_{X}^{p}(\log \Delta_{X}))  
  \end{equation}
  Here \(H_{\Phi}^{q }\) denotes local cohomology with
  respect to the family of supports \(\Phi_{X}\).  For connected
  \(X\), we \emph{define}
  \(H_{d}\lsps{X} : = H^{2 \dim X - d} \lsps{X}\),
  and in general we set
  \(H_{d} \lsps{X} = \bigoplus_{i} H_{d}\lsps{X_{i}}
  \) where \(X_{i}\) are the connected components of \(X\).
\end{definition}

Let \(f: \lsps{X} \to \lsps{Y}\) be pulling morphism of snc pairs with supports.

\begin{lemma}[{\cite[Prop. 2.3.1]{MR3838359} + \eqref{eq:ext-pow-om}}]
  \label{lem:morph-cxs-shvs-diff}
    The map \(f\) induces a morphism of  sheaves 
  \begin{equation}
    \label{eq:90}
    \begin{split}
      &f^{*} \Omega_{Y}^{p}(\log \Delta_{Y}) \xrightarrow{d \,
        f^{\vee}} \Omega_{X}^{p}(\log \Delta_{X}) \text{
        adjoint to a morphism  } \\
      &f^* \Omega_{Y}^{p} (\log \Delta_{Y}) \xrightarrow{d f^{\vee}}
      \Omega_{X}^{p}(\log \Delta_{X})  \text{ for all p.}
    \end{split}
  \end{equation}
\end{lemma}
The essential content of this lemma is that when we pull back a log
differential form \(\sigma \) on \((Y, \Delta_{Y})\), it doesn't
\emph{develop} poles of order \(\geq 1\) along \(\Delta_{X}\). Combining the previous lemma with proposition \ref{prp:coho-with-supp} gives:

\begin{proposition}[{\cite[\S
9.1-2]{bindaTriangulatedCategoriesLogarithmic2020}, see also \cite[\S 2.1]{MR2923726}}]
  \label{prop:pullb-pushf-morph}
  For every pulling morphism \(f: \lsps{X} \to \lsps{Y}\)  there are functorial
  morphisms
  \begin{equation}
    \label{eq:54}
    R \underline{\Gamma}_{\Phi} \ologd{Y}{p} \to R f_{*}
    R \underline{\Gamma}_{\Phi} \ologd{Y}{p} \text{  for all p  }
  \end{equation}
  In particular, for each \(p, q\) there are functorial homomorphisms
  \begin{equation}
    \label{eq:195}
    f^{*} : H_{\Phi}^{q}(Y, \ologd{Y}{p}) \to H_{\Phi}^{q}(X, \ologd{X}{p})
  \end{equation}
  and hence (summing over \(p +  q = d\)) functorial homomorphisms
  \begin{equation}
    \label{eq:196}
    f^{*} : H^{d}\lsps{X} \to H^{d} \lsps{Y}
  \end{equation}
\end{proposition}

The maps \(f_{*}: H_{d}\lsps{X} \to H_{d}\lsps{Y}\) induced by a pushing
morphism \(f: \lsps{X} \to \lsps{Y}\) can be obtained from a combination of
Nagata compactification and Grothendieck duality.

\begin{lemma}[{\cite[\S 9.5]{bindaTriangulatedCategoriesLogarithmic2020}, see
also \cite[\S 2.3]{MR2923726}}]
  \label{lem:pushforward}
  Let \(f: \lsps{X} \to \lsps{Y}\) be a pushing morphism of equidimensional
  log-smooth pairs with support such that. Then letting \(c = \dim Y - \dim X\),
  for each \(p\) there are functorial morphisms of complexes of coherent sheaves
  \begin{equation}
    \label{eq:32}
    R f_{*} R \underline{\Gamma}_{\Phi_X} (\Omega_{X}^{p}(\log \Delta_{X})) \to R \underline{\Gamma}_{\Phi_Y}\Omega_{Y}^{p+c}(\log    \Delta_{Y})[c] 
  \end{equation}
  inducing maps on cohomology
  \begin{equation}
    \label{eq:33}
    f_{*}: H_{\Phi_X}^{q}(X, \Omega_{X}^{p}(\log \Delta_{X})) \to H_{\Phi_Y}^{q+c}(Y, \Omega_{Y}^{p +      c}(\log \Delta_{Y})) 
  \end{equation}
  for all \(q\). 
\end{lemma}

Since they enter into the calculations below, we give a description of these
pushforward morphisms. Before beginning, a word on duality in our current setup:
since we are working exclusively over \(\Spec k\), we can make use of compatible
normalized dualizing complexes --- namely, if \(\pi: Z \to \Spec k\) is a
separated finite type \(k\)-scheme then \(\pi^! \strshf{\Spec \, k}\) is a dualizing
complex \cite[\href{https://stacks.math.columbia.edu/tag/0E2S}{Tag 0E2S},
\href{https://stacks.math.columbia.edu/tag/0FVU}{Tag 0FVU}]{stacks-project}.
We will make repeated use of the behavior of dualizing with respect to
differentials: as a consequence of \cref{lem:loc-descr-diff-log-poles}, wedge product gives a
perfect pairing 
\begin{equation}
  \label{eq:ppp-pair}
  \Omega^p_X(\log \Delta_X)(-\Delta_X) \otimes \Omega^{\dim X -p}_X(\log
\Delta_X) \to \omega_X 
\end{equation}
(see also \cite[Cor. III.7.13]{MR0463157}) and so \( \Omega^{\dim X -p}_X(\log
\Delta_X) \simeq R\sHom_X (\Omega^p_X(\log \Delta_X)(-\Delta_X), \omega_X )\). Here the derived sheaf Hom \(R\sHom_X\) agrees with the regular sheaf Hom as
\( \Omega^p_X(\log \Delta_X)(-\Delta_X)\) is locally free. On
the other hand, the \emph{dualizing functor} of \(X\) is  \(R\sHom_X
(\Omega^p_X(\log \Delta_X)(-\Delta_X), \omega_X[\dim  X] )\) where
\(\omega_X = \Omega_X^{\dim X}\). An upshot is that Grothendieck duality
calculations involving the sheaves of differential forms become more symmetric
and predictable if we work with the shifted versions \( \Omega^p_X(\log
\Delta_X)(-\Delta_X)[p] \); for example then we have the identity 
\[ \Omega^{\dim X -p}_X(\log \Delta_X)[\dim X -p] \simeq R\sHom_X
(\Omega^p_X(\log \Delta_X)(-\Delta_X)[p], \omega_X[\dim X] ) \]

Now, we need to compactify \(f: X \to Y\). 
\begin{theorem}[{\cite[\S 4 Thm. 2]{MR158892}, \cite[Thm. 4.1]{MR2356346}}]
  \label{thm:nagata}
  Let \(S\) be a quasi-compact quasi-separated scheme and let \(X \to S\) be a
  separated morphism of finite type. Then there is a dense open immersion of
  \(S\)-schemes \(X \inj \overline{X}\) such that \(\overline{X}\) is proper.
\end{theorem}

Using \cref{thm:nagata} we obtain morphisms of schemes 
\begin{equation}
  \label{eq:nagata}
  \begin{tikzcd}
    X \arrow[r, hook, "\iota"] \arrow[dr, "f"] & \bar{X} \arrow[d, "\bar{f}"] \\
    & Y 
  \end{tikzcd}
\end{equation}
where \(\iota: X \to \bar{X}\) is a dense open immersion and  \(\bar{f}: \bar{X} \to Y\) is \emph{proper}. Note
that \(\bar{X}\) need not be smooth over \(k\), and in the absence of
resolutions of singularities\footnote{At the time of this writing, this applies
to the cases \(\mathrm{char} \, k = p > 0\) and \(\dim X > 3\).} there is not even a way
to make \(\bar{X}\) smooth. This means we cannot hope to upgrade \(\bar{X}\) to
a simple normal crossing pair \(\lsp{\bar{X}}\). However, we do still have a
divisor \(\Delta_{\bar{X}} := \bar{f}^* \Delta_y\) on \(\bar{X}\). One way to
overcome these difficulties is to equip the possibly singular \(\bar{X}
\) with a \emph{logarithmic structure}, in some sense associated to \(\Delta_{\bar{X}}\), whose restriction to \(X\) coincides with
a logarithmic structure naturally defined by the simple normal crossing divisor
\(\Delta_X \). 

Formally, we use the log structure on \(\bar{X}\) pulled back from the
log structure on \(\lsp{Y}\) \cite[\S III.1.6-7]{MR3838359} along
the morphism \(\bar{f}: \bar{X} \to Y\). Since \((Y, \Delta_Y = \sum_{i=1}^N D^Y_i) \) is a simple normal
crossing pair, its associated log structure is Deligne-Faltings \cite[\S
III.1.7]{MR3838359} and can be encoded in the sequence of inclusions of ideal
sheaves \(\strshf{Y}(- D^Y_i) \inj \strshf{Y}\). The pullback log structure on
\(\bar{X}\) can then be encoded in the sequence of inclusions of ideal
sheaves \[\bar{f}^{-1}\strshf{Y}(- D^Y_i)\cdot \strshf{\bar{X}} =
\strshf{\bar{X}}(-\bar{f}^* D^Y_i) \inj \strshf{\bar{X}}.\]


The pushforward morphisms of \cref{lem:pushforward} are defined using the
sheaves of log differential \(p\)-forms on \(\bar{X}\) over \(k\) as described
in \cite[\S IV.1, V.2]{MR3838359} --- these will be denoted\footnote{This is an
abuse of notation since the construction of this sheaf is (as far as we know)
not the same as the one for simple normal crossing pairs described above
\cref{lem:loc-descr-diff-log-poles}, however the notation of \cite{MR3838359}
seems unsatisfactory for our purposes as we wish to stress that
these are not the ordinary differential forms \(\Omega^p_{\overline{X}}\),} by
\(\Omega^p_{\overline{X}}(\log \Delta_{\overline{X}})\). The
essential properties that we need are:
\begin{itemize}
  \item \(\Omega^p_{\overline{X}}(\log \Delta_{\overline{X}})\) is a
  coherent sheaf on \(\overline{X}\) together with a functorial
  morphism \[\Omega^p_Y(\log \Delta_Y) \to \overline{f}_*
  \Omega^p_{\overline{X}}(\log \Delta_{\overline{X}}).\]
  Coherence can be obtained as follows: first, the log structure on \(\lsp{Y}\)
  is coherent (\cite[\S III.1.9]{MR3838359}), and hence so is its pullback to
  \(\bar{X}\) (see for example \cite[Def. III.1.1.5, Rmk III.1.1.6]{MR3838359}).
  Then \cite[Cor. IV.1.2.8]{MR3838359} implies \(\Omega^1_{\overline{X}}(\log
  \Delta_{\overline{X}})\) is a coherent sheaf, and it follows that its \(p\)-th
  exterior powers are coherent sheaves as well. The desired functorial morphism
  can be obtained from \cite[Prop. IV.1.2.15]{MR3838359}.
  \item There is a natural isomorphism \(\Omega^p_{\overline{X}}(\log
  \Delta_{\overline{X}})|_X \simeq \Omega^p_{X }(\Delta_X)\). This can be seen
  by observing that the log structures on \(\lsp{X} \) and \(\bar{X}\) are
  obtained as pullbacks of the log structure on \(\lsp{Y}\) with respect to
  \(f\) and \(\bar{f} \) respectively (in the case of \(\lsp{X}\) this follows
  from \cref{def:push-pull}, and in the latter case  it is how we defined the
  log structure on  \(\bar{X}\)). Hence considering \cref{eq:nagata} we find
  that the log structure on \(\bar{X}\) restricts to that on \(\lsp{X}\).
\end{itemize}
Hence in particular \(\Omega^p_{\overline{X}}(\log \Delta_{\overline{X}})\) is a
\emph{functorial coherent extension} of \(\Omega^p_{X }(\Delta_X)\) to the
possibly non-snc log scheme \(\bar{X}\). Starting with the
log differential
\[ d \pr_Y^\vee: \Omega^p_Y(\log \Delta_Y)[p] \to R\overline{f}_*
\Omega^p_{\overline{X}}(\log \Delta_{\overline{X}})[p], \] twisting by
\(-\Delta_Y\) and using the projection formula gives a morphism  (\emph{note}:
this is where we use the assumptions that \(f^* \Delta_Y = \Delta_{X}\) and \(\bar{f}^* \Delta_Y = \Delta_{\bar{X}}\))
\begin{equation}
  \label{eq:twisted-diff}
  \Omega^p_Y(\log \Delta_Y)(-\Delta_Y)[p] \to R\overline{f}_*
\Omega^p_{\overline{X}}(\log \Delta_{\overline{X}})(-
\Delta_{\overline{X}})[p]
\end{equation} to which we apply Grothendieck duality:
\begin{theorem}[{Grothendieck duality, \cite[Cor. VII.3.4]{MR0222093},
  \cite[Thm. 3.4.4]{MR1804902}}]
  \label{thm:GD}
  Let \( f : X \to Y \) be a proper morphism of finite-dimensional noetherian
  schemes and assume \(Y\) admits a dualizing complex (for example \(X \) and
  \(Y \) could be schemes of finite type over \(k \)). Then for any pair of objects \(
  \mathscr{F}^\bullet \in D_{qc}^-(X)  \) and \(\sG^\bullet \in D_c^+(Y)\) 
  there is a natural isomorphism
  \[ Rf_* R \underline{Hom}_X(\mathscr{F}^\bullet, f^{!}\sG^\bullet) \simeq
    R\underline{Hom}_Y(Rf_* \mathscr{F}^\bullet, \sG^\bullet) \text{ in }
    D_c^b(Y) \]
\end{theorem}
Combining \cref{thm:GD} with \cref{eq:twisted-diff} gives a morphism 
\begin{equation}
  \small
  \begin{tikzcd}
    R\overline{f}_* R\sHom_{\overline{X} }(\Omega^p_{\overline{X} }(\log
  \Delta_{\overline{X}})(-\Delta_{\overline{X}})[p], \omega_{\overline{X}}^\bullet)  = R\sHom_Y
  (R\overline{f}_*\Omega^p_{\overline{X} }(\log \Delta_{\overline{X}})(-\Delta_{\overline{X}})[p] , \omega_Y[\dim Y] ) \arrow[d] \\
   R\sHom_Y
  (\Omega^p_Y(\log \Delta_Y)(-\Delta_Y)[p] , \omega_Y[\dim Y] ) 
  \end{tikzcd}
\end{equation}
where the equality is \cref{thm:GD} and the vertical map is induced by \eqref{eq:twisted-diff}.
Adding supports gives a morphism 
\begin{equation}
  \label{eq:pre-pushforward}
  {\footnotesize
  \begin{tikzcd}
    Rf_* R\underline{\Gamma}_{\Phi_X} R\sHom_{X }(\Omega^p_{X}(\log
  \Delta_{X})(-\Delta_{X})[p], \omega_{X}[\dim X]) = R\overline{f}_* R\underline{\Gamma}_{\Phi_X} R\sHom_{\overline{X} }(\Omega^p_{\overline{X} }(\log
  \Delta_{\overline{X}})(-\Delta_{\overline{X}})[p], \omega_{\overline{X}}^\bullet)  \arrow[d] \\
   R\underline{\Gamma}_{\Phi_Y} R\sHom_Y
  (\Omega^p_Y(\log \Delta_Y)(-\Delta_Y)[p] , \omega_Y[\dim Y] ) 
  \end{tikzcd}}
\end{equation}
where the equality is obtained from the \emph{excision} property of local
cohomology, compatibility of the dualizing functor with restriction \emph{and}
the natural isomorphism \(\Omega^p_{\overline{X}}(\log \Delta_{\overline{X}})|_X
\simeq \Omega^p_{X }(\Delta_X)\). 
Using \eqref{eq:ppp-pair} we obtain
\[ \Omega^{\dim X -p}_X(\log \Delta_X) \simeq \sHom_X(\Omega^p_X(\log
\Delta_X)(-\Delta_X),\omega_X ) = R\sHom_{X }(\Omega^p_{X}(\log
\Delta_{X})(-\Delta_{X}), \omega_{X})  \] where the last equality uses the fact
that \(\Omega^p_X(\log \Delta_X)(-\Delta_X)\) is locally free. A similar
calculation on \(Y\) transforms \eqref{eq:pre-pushforward} into:
\[ Rf_* R\underline{\Gamma}_{\Phi_X} \Omega^{\dim X -p}_X(\log \Delta_X)[\dim X
- p] \to  R\underline{\Gamma}_{\Phi_Y} \Omega^{\dim Y - p}_Y(\log \Delta_Y)[\dim Y - p]  \]
and reindexing like \(p \leftrightarrow \dim X - p  \) recovers \cref{lem:pushforward}.

\section{A base change formula}
\label{sec:proj-form}
\begin{lemma}[{compare with \cite[Prop. 2.3.7]{MR2923726}}]
  \label{lem:base-change-formula}
  Let 
  \begin{equation}
    \label{eq:projform-cartdiag}
    \begin{tikzcd}
      \lsps{X'} \arrow[dr, phantom, "\square"] \arrow[r, "g'"] \arrow[d, "f'"] & \lsps{X} \arrow[d, "f"] \\
      \lsps{Y'} \arrow[r, "g"] & \lsps{Y}
    \end{tikzcd}
  \end{equation}
  be a cartesian diagram of equidimensional snc pairs with supports, where \(f,
  f'\) (resp. \(g, g'\)) are pushing (resp. pulling) morphisms and \(g\) is
  either flat or a closed immersion transverse to \(f\). Then 
  \[ g^* f_* = f'_* g'^* : H^* \lsps{X} \to H^*\lsps{Y'}. \]
\end{lemma}
We will prove this following Chatzistamatiou and R\"ulling's argument
\cite[Prop. 2.3.7]{MR2923726} quite closely, at various points reducing to
statements proved therein. In the proofs we will make use of a slight variant of \cref{def:inv-image-supp}.
\begin{definition}
  \label{def:str-trans-supp}
  If \(f: X \to Y\) is a morphism of noetherian schemes and let \(\Phi_Y\) is a
  family of supports on \(Y\), then 
  \[ f^{-1}_* (\Phi_Y) :=  \{Z \subseteq X  \, | \, f|_{Z}
  \text{ is proper and } f(Z) \in \Phi_Y\}\]
\end{definition}

\begin{lemma}
  \label{lem:pr-or-climm}
  It suffices to prove \cref{lem:base-change-formula} in the cases where \(f\) is
  either 
  \begin{enumerate}
    \item \label{item:pr-or-climm1} a projection morphism of the form \(
    \mathrm{pr}_Y: (X\times Y, \mathrm{pr}_Y^* \Delta_Y, \mathrm{pr}_{Y*}^{-1}(\Phi_Y)) \to (Y, \Delta_Y,
    \Phi_Y)\), or
    \item \label{item:pr-or-climm2} a closed immersion.
  \end{enumerate}
\end{lemma}
\begin{remark}
  This lemma makes essential use of the \emph{functoriality} part of \cref{lem:pushforward}.
\end{remark}
\begin{proof}
  We can decompose \eqref{eq:projform-cartdiag} as a concatenation of cartisian diagrams
  \begin{equation}
    \label{eq:pform-decomp}
    \begin{tikzcd}
      \lsps{X'} \arrow[dr, phantom, "(2)"] \arrow[r, "g'"] \arrow[d, "h'"] & \lsps{X} \arrow[d, "h"] \\
      (X\times Y', \mathrm{pr}_{Y'}^* \Delta_Y, \mathrm{pr}_{Y'*}^{-1}(\Phi_Y')) \arrow[dr, phantom, "(1)"] \arrow[d, "\pr_{Y'}"] \arrow[r, "\mathrm{id} \times g"] &(X\times Y, \mathrm{pr}_Y^* \Delta_Y, \mathrm{pr}_{Y*}^{-1}(\Phi_Y)) \arrow[d, "\pr_Y"]\\
      \lsps{Y'} \arrow[r, "g"] & \lsps{Y}
    \end{tikzcd}
  \end{equation}
  where \(h = \mathrm{id} \times f\) is the graph morphism of \(f\) and \(h' =
  g' \times f'\). If \(g\) is flat or a closed immersion transverse to \(f\)
  then \(\mathrm{id} \times g\) is flat or a closed immersion transverse to \(h\) (by base change).
  
  Here the only new feature not covered in \cite[Prop. 2.3.7]{MR2923726} is the
  presence of divisors, and we simply note that \(\Delta_X = f^*\Delta_X = h^*
  \mathrm{pr}_Y^* \Delta_Y \) and similarly for \(\Delta_{X'}\), so that both
  \(\pr_Y\) and \(h\) are pushing morphisms in the sense of
  \cref{def:push-pull}, and similarly for the left vertical maps. In other
  words, the supports and divisors in the middle row have been chosen precisely
  so that the vertical morphisms are all ``pushing.''
\end{proof}

We proceed to consider case \cref{item:pr-or-climm1}, and wish to point out that
for this case \(g\) can be arbitrary (we will need the flatness/transversality restrictions in case \cref{item:pr-or-climm2}). In what follows we set
\(d_X = \dim X, d_Y = \dim Y\) and similarly for \(X', Y'\).  
Using \cref{thm:nagata} we
obtain a compactification \(\iota: X \inj \overline{X}\) over \(k\) of the
smooth, separated and finite type \(k\)-scheme \(X\) in the upper right corner
of \eqref{eq:projform-cartdiag} and \eqref{eq:pform-decomp}. This results in a
compactification of the square (1) in \eqref{eq:pform-decomp} which we write as 
\begin{equation}
  \label{eq:cpctified-pform}
  \begin{tikzcd}
    (X \times Y', \mathrm{pr}_{Y'}^* \Delta_Y, \mathrm{pr}_{Y'*}^{-1}(\Phi_Y')) \arrow[d, "\iota \times \mathrm{id}"] \arrow[r, "\mathrm{id} \times g"] &(X \times Y, \mathrm{pr}_Y^* \Delta_Y, \mathrm{pr}_{Y*}^{-1}(\Phi_Y)) \arrow[d, "\iota \times \mathrm{id}"] \\
    (\overline{X} \times Y', \overline{\pr}_{Y'}^* \Delta_Y, \overline{\pr}_{Y'*}^{-1}(\Phi_Y')) \arrow[d, "\overline{\pr}_{Y'}"] \arrow[r, "\mathrm{id} \times g"] &(\overline{X} \times Y, \overline{\pr}_Y^* \Delta_Y, \overline{\pr}_{Y*}^{-1}(\Phi_Y)) \arrow[d, "\overline{\pr}_Y"]\\
      \lsps{Y'} \arrow[r, "g"] & \lsps{Y}
  \end{tikzcd}
\end{equation}
By the description following \cref{lem:pushforward}, we know that 
\[\pr_{Y*} : H^*(X \times Y, \mathrm{pr}_Y^* \Delta_Y,
\mathrm{pr}_{Y*}^{-1}(\Phi_Y))  \to H^*\lsps{Y} \] stems from a morphism 
\begin{equation}
  \label{eq:map-from-which-pushfor-stems}
  R\overline{\pr}_{Y*} R\sHom_{\overline{X} \times Y }(\Omega^p_{\overline{X} \times Y}(\log
  \mathrm{pr}_Y^* \Delta_Y )(-\mathrm{pr}_Y^* \Delta_Y)[p], \omega_{\overline{X} \times Y}^\bullet)  \to \Omega^{d_Y - p}_Y(\log \Delta_Y)[d_Y - p] 
\end{equation}
obtained as the Grothendieck dual of a log differential of \(\overline{\pr}_Y\)
(here and throughout what follows, a similar statement holds for
\(\overline{\pr}_{Y'}\)). By an observation of Chatzistamatiou-R\"ulling
, this map factors as 
\begin{equation}
  \label{eq:cr-factorization}
  \begin{split}
    &R\overline{\pr}_{Y*} R\sHom_{\overline{X} \times Y }(\Omega^p_{\overline{X} \times Y}(\log\overline{\pr}_Y^* \Delta_Y )(-\overline{\pr}_Y^* \Delta_Y)[p], \omega_{\overline{X} \times Y}^\bullet) \\
     &\to R\overline{\pr}_{Y*} R\sHom_{\overline{X} \times Y }(L \overline{\pr}_{Y}^* \Omega^p_Y(\log \Delta_Y)(-\Delta_Y)[p], \omega_{\overline{X} \times Y}^\bullet)\\
    &\xrightarrow[\text{adjunction}]{\simeq}  R\sHom_Y
    (\Omega^p_Y(\log \Delta_Y)(-\Delta_Y)[p] , R\overline{\pr}_{Y*} \omega_{\overline{X} \times Y}^\bullet ) \\
    & \xrightarrow[\text{trace}]{} R\sHom_Y
    (\Omega^p_Y(\log \Delta_Y)(-\Delta_Y)[p] , \omega_Y[d_Y])\\
    &\xrightarrow[]{\simeq} \Omega^{d_Y - p}_Y(\log \Delta_Y)[d_Y - p] 
  \end{split}
\end{equation}
where the adjunction isomorphism is \cite[Prop. II.5.10]{MR0222093}, and the map
labeled trace is induced by the Grothendieck trace \(R\overline{\pr}_{Y*}
\omega_{\overline{X} \times Y}^\bullet \to \omega_Y[d_Y]\). \emph{If} it were
the case that \(\overline{X}\) were smooth, then the usual ``box product''
decomposition 
\[\omega_{\overline{X} \times Y}^\bullet \simeq \omega_{\overline{X}}[d_X]
\boxtimes \omega_Y[d_Y] := \pr_{\overline{X}}^* \omega_{\overline{X}}[d_X] \otimes
\overline{\pr}_{Y*} \omega_Y[d_Y] \] together with the perect pairings
\eqref{eq:ppp-pair} and the local freeness of \(\Omega^p_Y(\log
\Delta_Y)(-\Delta_Y)[p]\) would give an identification 
\begin{equation}
  \small
  \label{eq:cr-crucial-lemma}
  R\sHom_{\overline{X} \times Y }(L \overline{\pr}_{Y}^* \Omega^p_Y(\log \Delta_Y)(-\Delta_Y)[p], \omega_{\overline{X} \times Y}^\bullet) \simeq \pr_{\overline{X}}^* \omega_{\overline{X}}[d_X] \otimes \overline{\pr}_{Y}^* \Omega^{d_Y - p}_Y(\log \Delta_Y)[d_Y - p]
\end{equation}
In fact a more careful version of this argument, carrying out the above calculation on the smooth locus \(X \times
Y\) and using excision, shows that \(H^*(X \times Y, \mathrm{pr}_Y^* \Delta_Y,
\mathrm{pr}_{Y*}^{-1}(\Phi_Y))  \to H^*\lsps{Y} \) \emph{always} factors through
the summand \(H_{\Phi_X}^* (X \times Y, \pr_{\overline{X}}^*
\omega_{\overline{X}} \otimes \overline{\pr}_{Y}^* \Omega^{d_Y - p}_Y(\log
\Delta_Y))\).

Our next lemma implies that even when \(\overline{X}\) is not known to be
smooth, \eqref{eq:map-from-which-pushfor-stems} still factors through something
like \(R\overline{\pr}_{Y*} (\pr_{\overline{X}}^* \omega_{\overline{X}}[d_X]
\otimes \overline{\pr}_{Y}^* \Omega^{d_Y - p}_Y(\log
\Delta_Y)[d_Y - p]) \), provided we replace \(\pr_{\overline{X}}^*
\omega_{\overline{X}}[d_X]\) with \(\overline{\pr}_{Y}^{!}\sO_Y\). 

\begin{lemma}[{compare with \cite[Lem. 2.2.16]{MR2923726}}]
  \label{lem:shriek-thing}
  For each \(p\) there is a natural map 
  \[ \gamma :  \overline{\pr}_{Y}^{!}\sO_Y \otimes \overline{\pr}_{Y}^*
  \Omega^{d_Y - p}_Y(\log \Delta_Y)(-\Delta_Y)[d_Y - p] \to
  R\sHom_{\overline{X} \times Y }(\overline{\pr}_{Y}^* \Omega^p_Y(\log
  \Delta_Y)(-\Delta_Y)[p], \omega_{\overline{X} \times Y}^\bullet) \] 
  such that the restriction of \(\gamma\) to \(X \times Y\) agrees with the
  isomorphism 
  \[ \pr_{X}^* \omega_{X}[d_X] \otimes
  \pr_{Y}^* \Omega^{d_Y - p}_Y(\log \Delta_Y)(-\Delta_Y)[d_Y -
  p]  \xrightarrow[]{\simeq} R\sHom_{X \times Y }(L
  \pr_{Y}^* \Omega^p_Y(\log \Delta_Y)(-\Delta_Y)[p],
  \omega_{X \times Y}^\bullet) \]
  and such that the composition
  \begin{equation}
    \label{eq:shriek-thing1}
    \begin{split}
      & R \overline{\pr}_{Y*}(\pr_{X}^* \omega_{X}[d_X] \otimes
      \pr_{Y}^* \Omega^{d_Y - p}_Y(\log \Delta_Y)(-\Delta_Y)[d_Y -
      p])  \\
      & \xrightarrow[]{R \overline{\pr}_{Y*}(\gamma)} R \overline{\pr}_{Y*}R\sHom_{X \times Y }(
      \pr_{Y}^* \Omega^p_Y(\log \Delta_Y)(-\Delta_Y)[p],
      \omega_{X \times Y}^\bullet) \\
      &\xrightarrow[\text{adjunction}]{\simeq} R\sHom_{X \times Y }(\Omega^p_Y(\log \Delta_Y)(-\Delta_Y)[p],
      R \overline{\pr}_{Y*} \omega_{X \times Y}^\bullet) \\
      &\xrightarrow[]{\text{trace}} R\sHom_{X \times Y }(\Omega^p_Y(\log \Delta_Y)(-\Delta_Y)[p], \omega_Y[d_Y]) \simeq \Omega^{d_Y - p}_Y(\log \Delta_Y)(-\Delta_Y)[d_Y - p]
    \end{split}
  \end{equation}
   coincides with the composition 
  \begin{equation}
    \label{eq:shriek-thing2}
    \begin{split}
      &R \overline{\pr}_{Y*}( \overline{\pr}_{Y}^{!}\sO_Y \otimes \overline{\pr}_{Y}^*  \Omega^{d_Y - p}_Y(\log \Delta_Y)(-\Delta_Y)[d_Y - p]  )\\
      &\xrightarrow[\text{form.}]{\text{proj.}} R \overline{\pr}_{Y*}( \overline{\pr}_{Y}^{!}\sO_Y  )\otimes \overline{\pr}_{Y}^*  \Omega^{d_Y - p}_Y(\log \Delta_Y)(-\Delta_Y)[d_Y - p] \\
      & \xrightarrow[]{\tr \otimes \mathrm{id}} \Omega^{d_Y - p}_Y(\log \Delta_Y)(-\Delta_Y)[d_Y - p]
    \end{split}
  \end{equation}
\end{lemma}
By base change for dualizing complexes
(\cite[\href{https://stacks.math.columbia.edu/tag/0BZX}{Tag
0BZX},\href{https://stacks.math.columbia.edu/tag/0E2S}{Tag
0E2S}]{stacks-project}) applied to the cartesian diagram
\[\begin{tikzcd}
  \overline{X} \times Y  \arrow[d] \arrow[r] &\overline{X} \arrow[d]\\
  Y \arrow[r] & \Spec k\\
\end{tikzcd}\]
(note that this is a very mild situation: \(\overline{X} \to \Spec k \) is flat
and proper and \(Y \to \Spec k\) is smooth) we see that
\(\overline{\pr}_{Y}^{!}\sO_Y \simeq \pr_{\overline{X}}^*
\omega_{\overline{X}}^\bullet\). This makes the map \(\gamma\) look even more
like \eqref{eq:cr-crucial-lemma}.
\begin{proof}
  Following \cite[Lem. 2.2.16]{MR2923726} we begin with the morphism 
  \[ e : \overline{\pr}_{Y}^{!}\sO_Y \otimes^L L\overline{\pr}_{Y}^{*}
  \omega_Y^\bullet \to \overline{\pr}_{Y}^{!} \omega_Y^\bullet  =:
  \omega_{\overline{X} \times Y}^\bullet
  \] of \cite[4.3.12]{MR1804902}, which as explained therein  agrees with 
  \[ \pr_X^* \omega_X [d_X] \otimes \pr_Y^* \omega_Y[d_Y]
  \xrightarrow[]{\simeq} \omega_{X \times Y} [d_X + d_Y]\] on locus \(X
  \times Y\),\footnote{See Conrad's comment ``It is easy to check
  that \(e_f\) coincides with (3.3.21) in the smooth case and is compatible with
  composites in f (using (4.3.6).''} and has the property that 
  \[ 
    \begin{tikzcd}
      R \overline{pr}_{Y*}( \overline{\pr}_{Y}^{!}\sO_Y \otimes^L L\overline{\pr}_{Y}^{*}  \omega_Y^\bullet) \arrow[r, "R \overline{pr}_{Y*} e"] \arrow[d, "\text{proj. form}"] & R \overline{pr}_{Y*} \omega_{\overline{X} \times Y}^\bullet \arrow[d, "\tr"] \\
      R \overline{pr}_{Y*} \overline{\pr}_{Y}^{!}\sO_Y \otimes^L \omega_Y^\bullet \arrow[r, "\tr \otimes \mathrm{id}"] & \omega_Y^\bullet \\
    \end{tikzcd}  
  \] commutes \cite[Thm. 4.4.1]{MR1804902}. We then define our version of \(\gamma\) as the composition 
  \begin{equation}
    \label{eq:my-gamma}
    \begin{split}
      &\overline{\pr}_{Y}^{!}\sO_Y \otimes^L L\overline{\pr}_{Y}^*
      \Omega^{d_Y - p}_Y(\log \Delta_Y)(-\Delta_Y)[d_Y - p] \\
      & \xrightarrow{\mathrm{id} \otimes^L \text{\eqref{eq:ppp-pair}}} \overline{\pr}_{Y}^{!}\sO_Y \otimes^L L\overline{\pr}_{Y}^*  R\sHom_Y (\Omega^{p}_Y(\log \Delta_Y)[p], \omega_Y^\bullet) \\
      & \xrightarrow[\text{of } L\overline{\pr}_{Y}^*, \otimes^L ]{\text{functoriality}} R\sHom_{\overline{X} \times Y}(L\overline{\pr}_{Y}^* \Omega^{p}_Y(\log \Delta_Y)[p], \overline{\pr}_{Y}^{!}\sO_Y \otimes^L \omega_Y^\bullet) \\
      &\xrightarrow[e]{\text{induced by}} R\sHom_{\overline{X} \times Y}(L\overline{\pr}_{Y}^* \Omega^{p}_Y(\log \Delta_Y)[p], \omega_{\overline{X} \times Y}^\bullet)
    \end{split}
  \end{equation}
  Note that we may drop the ``\(L\)''s as \(\Omega^{d_Y - p}_Y(\log
  \Delta_Y)(-\Delta_Y)\) and \(\Omega^{p}_Y(\log \Delta_Y)\) are locally free.
  Verification of the stated compatibilities is as in \cite[Lem.
  2.2.16]{MR2923726}.
\end{proof}
\begin{remark}
  It seems like we could have also used the more general version of \cite[4.3.12]{MR1804902}
  \[e': \overline{\pr}_{Y}^{!}\sO_Y \otimes^L L\overline{\pr}_{Y}^*
  \Omega^{d_Y - p}_Y(\log \Delta_Y)(-\Delta_Y)[d_Y - p] \to
  \overline{\pr}_{Y}^{!}\Omega^{d_Y - p}_Y(\log \Delta_Y)(-\Delta_Y)[d_Y - p] \]
  together with the description 
  \[\overline{\pr}_{Y}^{!}\Omega^{d_Y - p}_Y(\log \Delta_Y)(-\Delta_Y)[d_Y - p]
  = D_{\overline{X}\times Y}(L\overline{\pr}_{Y}^* D_Y (\Omega^{d_Y - p}_Y(\log
  \Delta_Y)(-\Delta_Y)[d_Y - p]))\]
  where \(D_Y (-) = R\sHom(- , \omega_Y^\bullet)\) and similarly for \(D_{\overline{X}\times Y}\).
\end{remark}

Using this modified \(\gamma\), we obtain a modified version of the diagram
\cite[p. 732 during Lem. 2.3.4]{MR2923726}, namely \eqref{eq:modified-732} in \cref{fig:modified-732}). To make this diagram legible, we use
a few abbreviations:  all functors are derived,  we use the
dualizing functors of the form \(D_Y(-) = R\sHom_Y(-, \omega_Y^\bullet)\) and we
let \(d = d_X + d_Y\). \cref{lem:shriek-thing} shows that triangles
involving \(\gamma\) commute, and \eqref{eq:cr-factorization} gives
commutativity of the rest of the diagram. The usefulness of this diagram
is that by \emph{definition} beginning in the top left corner and following the
path \(\rightarrow \downarrow \) we obtain the pushforward on Hodge cohomology 
\[ \pr_{Y*} \underline{\Gamma}_{\pr_{Y*}^{-1}\Phi_Y} \Omega^{d - p}_{X
\times Y}(\log \mathrm{pr}_Y^* \Delta_Y)[d -p] \to
\underline{\Gamma}_{\Phi_Y}\Omega^{d_Y-p}_{ \times Y}(\log   \Delta_Y )(-
\Delta_Y)[d_Y-p]  \] but following \(\downarrow \rightarrow\) gives a
composition whose  behavior with respect to \eqref{eq:cpctified-pform} is easier
to analyze. Namely, we have a diagram like \eqref{eq:modified-732} on \(Y'\),
and in fact a map from \eqref{eq:modified-732} to \(g_*\) of the analogous
diagram on \(Y'\), and hence from the preceding discussion it will suffice to
prove commutativity of \eqref{eq:down-around} of \cref{fig:modified-732}.

\begin{sidewaysfigure}
  \centering
  \begin{equation}
    \scriptsize
    \label{eq:modified-732}
    \begin{tikzcd}[column sep=tiny, row sep=huge]
      \pr_{Y*} \underline{\Gamma}_{\pr_{Y*}^{-1}\Phi_Y} \Omega^{d - p}_{X \times Y}(\log \mathrm{pr}_Y^* \Delta_Y)[d -p] \arrow[rr, "\text{excision}"] \arrow[dd, "\text{excision} +\cref{lem:shriek-thing}"]  \arrow[dr] && \overline{\pr}_{Y*} \underline{\Gamma}_{\overline{\pr}_{Y*}^{-1}\Phi_Y} D_{\overline{X} \times Y }(\Omega^p_{\overline{X} \times Y}(\log    \mathrm{pr}_Y^* \Delta_Y )(-\mathrm{pr}_Y^* \Delta_Y)[p]) \arrow[dd, "d\overline{\pr}_{Y}^\vee"] \arrow[dl, "d\overline{\pr}_{Y}^\vee"] \\
      & \overline{\pr}_{Y*} \underline{\Gamma}_{\overline{\pr}_{Y*}^{-1}\Phi_Y} D_{\overline{X} \times Y }(\overline{\pr}_{Y}^*\Omega^p_{ \times Y}(\log   \Delta_Y )(- \Delta_Y)[p]) \arrow[dr] & \\
      \overline{\pr}_{Y*} \underline{\Gamma}_{\overline{\pr}_{Y*}^{-1}\Phi_Y}  (\overline{\pr}_{Y}^{!}\sO_Y \otimes \overline{\pr}_{Y}^*
      \Omega^{d_Y - p}_Y(\log \Delta_Y)(-\Delta_Y)[d_Y - p] )  \arrow[rr,"\eqref{eq:shriek-thing2}"] \arrow[ur, "\overline{\pr}_{Y*}(\gamma)"] && \underline{\Gamma}_{\Phi_Y} D_{Y}(\Omega^p_{  Y}(\log   \Delta_Y )(- \Delta_Y)[p]) = \underline{\Gamma}_{\Phi_Y}\Omega^{d_Y-p}_{ Y}(\log   \Delta_Y )(- \Delta_Y)[d_Y-p] 
    \end{tikzcd}
  \end{equation}  
  \vspace{5em}
  \begin{equation}
    \label{eq:down-around}
    \footnotesize
    \begin{tikzcd}[column sep=small, row sep=huge]
      \pr_{Y*} \underline{\Gamma}_{\overline{\pr}_{Y*}^{-1}\Phi_Y} \Omega^{d - p}_{X \times Y}(\log \mathrm{pr}_Y^* \Delta_Y)[d -p] \arrow[r] \arrow[d]& \overline{\pr}_{Y*} \underline{\Gamma}_{\overline{\pr}_{Y*}^{-1}\Phi_Y}  (\overline{\pr}_{Y}^{!}\sO_Y \otimes \overline{\pr}_{Y}^*
      \Omega^{d_Y - p}_Y(\log \Delta_Y)(-\Delta_Y)[d_Y - p] ) \arrow[r] \arrow[d] &   \underline{\Gamma}_{\Phi_Y}\Omega^{d_Y-p}_{ Y}(\log   \Delta_Y )(- \Delta_Y)[d_Y-p] \arrow[d] \\
      g_*\pr_{Y'*} \underline{\Gamma}_{\overline{\pr}_{Y'*}^{-1}\Phi_{Y'}} \Omega^{d - p}_{X \times Y'}(\log \mathrm{pr}_{Y'}^* \Delta_{Y'})[d -p] \arrow[r] &  g_*\overline{\pr}_{Y'*} \underline{\Gamma}_{\overline{\pr}_{Y'*}^{-1}\Phi_{Y'}}  (\overline{\pr}_{Y'}^{!}\sO_{Y'} \otimes \overline{\pr}_{Y'}^*
      \Omega^{d_Y - p}_{Y'}(\log \Delta_{Y'})(-\Delta_Y)[d_Y - p] ) \arrow[r]  &   g_*\underline{\Gamma}_{\Phi_{Y'}}\Omega^{d_Y-p}_{ Y'}(\log   \Delta_{Y'} )(- \Delta_{Y'})[d_Y-p]
    \end{tikzcd}
  \end{equation}
  \caption[{Modified versions of diagrams appearing in the proof of \cite[Lem.
  2.3.4]{MR2923726}}]{Modified versions of diagrams appearing in the proof of
  \cite[Lem.  2.3.4]{MR2923726} (all functors derived)}
  \label{fig:modified-732}
\end{sidewaysfigure}

Applying excision together with \cref{lem:shriek-thing} we may rewrite the top
row of \eqref{eq:down-around} as 
\begin{equation}
  \label{eq:rewrite-the-rows}
  \begin{split}
    &R\pr_{Y*} R\underline{\Gamma}_{\pr_{Y*}^{-1}\Phi_Y} \Omega^{d - p}_{X \times Y}(\log \mathrm{pr}_Y^* \Delta_Y)[d -p] \\
    & \xrightarrow{\text{project}} R\pr_{Y*} R\underline{\Gamma}_{\pr_{Y*}^{-1}\Phi_Y}  (\pr_X^* \omega_X[d_X] \otimes \pr_Y^* \Omega^{d_Y - p}_Y(\log \Delta_Y)(-\Delta_Y)[d_Y - p] )   \\
    &\xrightarrow[\text{form.}]{\text{proj.}} R\pr_{Y*} R\underline{\Gamma}_{\pr_{Y*}^{-1}\Phi_Y}  (\pr_X^* \omega_X[d_X]  ) \otimes  \Omega^{d_Y - p}_Y(\log \Delta_Y)(-\Delta_Y)[d_Y - p] \\
    & \xrightarrow[]{\tr \otimes \mathrm{id}} R\underline{\Gamma}_{\Phi_Y}\Omega^{d_Y-p}_{ Y}(\log   \Delta_Y )(- \Delta_Y)[d_Y-p]
  \end{split}
\end{equation}
where the first map is induced by a projection 
\[\Omega^{d - p}_{X \times
Y}(\log \mathrm{pr}_Y^* \Delta_Y)[d -p] \to \pr_X^* \omega_X[d_X] \otimes
\pr_Y^* \Omega^{d_Y - p}_Y(\log \Delta_Y)(-\Delta_Y)[d_Y - p]\]
coming from a
K\"unneth-type decomposition of \(\Omega^{d - p}_{X \times Y}(\log
\mathrm{pr}_Y^* \Delta_Y)\), the second is the projection formula, and the last
map is induced by a trace map with supports defined as the composition
\begin{equation}
  \label{eq:trace-w-supports}
  \begin{split}
    & R\pr_{Y*} R\underline{\Gamma}_{\pr_{Y*}^{-1}\Phi_Y}  (\pr_X^* \omega_X[d_X]  )\xrightarrow{\text{excision}} R\overline{\pr}_{Y*} R\underline{\Gamma}_{\overline{\pr}_{Y*}^{-1}\Phi_Y}  (\overline{\pr}_{Y}^{!}\sO_Y) \\
    & \xrightarrow[]{\text{\cref{prp:coho-with-supp}}} R\underline{\Gamma}_{\Phi_Y} R\overline{\pr}_{Y*}   (\overline{\pr}_{Y}^{!}\sO_Y)\xrightarrow{\tr} R\underline{\Gamma}_{\Phi_Y}  \sO_Y
  \end{split}
\end{equation}
Here the second map comes from the functoriality properties of
\cref{prp:coho-with-supp}, since there is an inclusion \(\pr_{Y*}^{-1}\Phi_Y
\subseteq \pr_{Y}^{-1}\Phi_Y\). The decomposition \eqref{eq:rewrite-the-rows}
maps to a similar decomposition of the bottom row of \eqref{eq:down-around}, and
the only commutativity not guaranteed by standard functoriality properties (e.g.
functoriality of the projection formula appearing in the second map of
\eqref{eq:rewrite-the-rows}) is that of
\begin{equation}
  \label{eq:nontriv-comm-traces}
  {\footnotesize
  \begin{tikzcd}[column sep=small]
    R\pr_{Y*} R\underline{\Gamma}_{\pr_{Y*}^{-1}\Phi_Y}  (\pr_X^* \omega_X[d_X]  ) \otimes  \Omega^{d_Y - p}_Y(\log \Delta_Y)(-\Delta_Y)[d_Y - p] \arrow[r, "\tr \otimes \mathrm{id}"] \arrow[d]  & R\underline{\Gamma}_{\Phi_Y}\Omega^{d_Y-p}_{ Y}(\log   \Delta_Y )(- \Delta_Y)[d_Y-p] \arrow[d] \\
    Rg_* (R\pr_{Y'*} R\underline{\Gamma}_{\pr_{Y'*}^{-1}\Phi_{Y'}}  (\pr_X^* \omega_X[d_X]  ) \otimes  \Omega^{d_Y - p}_{Y'}(\log \Delta_{Y'})(-\Delta_{Y'})[d_Y - p]) \arrow[r, "\tr' \otimes \mathrm{id}"] & Rg_*(R\underline{\Gamma}_{\Phi_{Y'}}\Omega^{d_Y-p}_{Y'}(\log   \Delta_{Y'} )(- \Delta_{Y'})[d_Y-p])
  \end{tikzcd}}
\end{equation}
But applying one more projection formula to the bottom row of
\eqref{eq:nontriv-comm-traces}, we see \eqref{eq:nontriv-comm-traces} is
obtained by tensoring the differential 
\[\Omega^{d_Y-p}_{ Y}(\log   \Delta_Y )(-
\Delta_Y)[d_Y-p] \to Rg_*\Omega^{d_Y-p}_{Y'}(\log   \Delta_{Y'} )(-
\Delta_{Y'})[d_Y-p]\]
with 
\begin{equation}
  \label{eq:known-comm-traces}
  \small
  \begin{tikzcd}[column sep=small]
    R\pr_{Y*} R\underline{\Gamma}_{\pr_{Y*}^{-1}\Phi_Y}  (\pr_X^* \omega_X[d_X]  ) \arrow[r, "\tr \otimes \mathrm{id}"] \arrow[d]  & R\underline{\Gamma}_{\Phi_Y}\sO_Y \arrow[d] \\
    Rg_* (R\pr_{Y'*} R\underline{\Gamma}_{\pr_{Y'*}^{-1}\Phi_{Y'}}  (\pr_X^* \omega_X[d_X]  )) \arrow[r, "\tr' \otimes \mathrm{id}"] & Rg_*(R\underline{\Gamma}_{\Phi_{Y'}}\sO_{Y'}) 
  \end{tikzcd}
\end{equation}
and the commutativity of \eqref{eq:known-comm-traces} is proved in \cite[Lem.
2.3.4]{MR2923726}. So far we have proved:

\begin{lemma}
  \Cref{lem:base-change-formula} holds in case \cref{item:pr-or-climm1} of \cref{lem:pr-or-climm}.
\end{lemma}

It remains to deal with case \cref{item:pr-or-climm2} of
\cref{lem:pr-or-climm}, and for this we use the following lemma.

\begin{lemma}[{compare with \cite[Cor. 2.2.22]{MR2923726}}]
  \label{lem:climms}
  Consider a diagram of pure-dimensional snc pairs 
  \begin{equation}
    \label{eq:cart-diag-climms}
    \begin{tikzcd}
      \lsp{X'} \arrow[r, "g'"] \arrow[d, "\imath'"] & \lsp{X} \arrow[d,
      "\imath"]\\
      \lsp{Y'} \arrow[r, "g"] &\lsp{Y}
    \end{tikzcd}
  \end{equation}
  where \(\imath, \imath'\) are pushing closed immersions and \(\dim Y - \dim X = \dim Y' - \dim X' =:c\).  Then, for all \(q \) the diagram 
  \begin{equation}
    \label{eq:proj-form-climm-case}
    \begin{tikzcd}
      \imath_* \Omega_X^q(\log \Delta_X) [q] \arrow[r, "dg'^{\vee}"] \arrow[d] & Rg_* \imath'_* \Omega_{X'}^q(\log \Delta_{X'}) \arrow[d] \\
      \Omega_{Y}^{q+c}(\log \Delta_Y )[q+c] \arrow[r, "dg^\vee"] & Rg_* \Omega_{Y'}^{q+c}(\log \Delta_{Y'})[q+c]
    \end{tikzcd}
  \end{equation}
  commutes, where the horizontal maps are induced by log differentials and the
  left vertical map is the composition
  \begin{equation}
    \label{eq:pushfor-climm}
    \begin{split}
      &\imath_* \Omega_X^q(\log \Delta_X) [q]  \xrightarrow{\simeq} \imath_* R\sHom(\Omega_X^{d_X -q}(\log \Delta_X)(-\Delta_X) [d_X -q], \omega_X^\bullet) \\
      &\xrightarrow{\text{duality}} R\sHom(\imath_* \Omega_X^{d_X -q}(\log \Delta_X)(-\Delta_X) [d_X -q], \omega_Y^\bullet)\\
      & \xrightarrow{d\imath^\vee} R\sHom( \Omega_Y^{d_X -q}(\log \Delta_Y)(-\Delta_Y) [d_X -q], \omega_Y^\bullet) \\
      &\xrightarrow{\simeq} \Omega_{Y}^{q+c}(\log \Delta_Y )[q+c]
    \end{split}
  \end{equation} 
  and the right vertical arrow is \(Rg_*\) of a similar composition on \(Y'\).
\end{lemma}
Note that the codimension hypotheses hold if \(g\) is flat or a closed immersion
transverse to \(\imath\).
\begin{proof}
  While it seems a proof following \cite[Cor. 2.2.22]{MR2923726} step-by-step is
  possible, we instead \emph{reduce} to the case proved there as follows: first, observe that there is an evident map from the
  cartesian diagram 
  \begin{equation}
    \label{eq:climms-interiors}
    \begin{tikzcd}
      U_{X'} \arrow[r] \arrow[d] & U_X \arrow[d] \\
      U_{Y'}  \arrow[r] & U_Y \\
    \end{tikzcd}
  \end{equation}
  of interiors to \eqref{eq:cart-diag-climms}. Noting that \eqref{eq:proj-form-climm-case} will map to a
  similar diagram obtained from  \eqref{eq:climms-interiors}, that the compositions \eqref{eq:pushfor-climm} are at least
  compatible with Zariski localization, \emph{and} that the situation of
  \eqref{eq:climms-interiors} is covered by \cite[Cor. 2.2.22]{MR2923726}, it
  will suffice to show that the natural map 
  \begin{equation}
    \label{eq:lazy-injection}
    h^0 R\sHom_Y(\imath_* \Omega_X^q(\log \Delta_X) [q],  Rg_*
  \Omega_{Y'}^{q+c}(\log \Delta_{Y'})[q+c]) \to  h^0 R\sHom_{U_Y}(\imath_*
  \Omega_{U_X}^q [q],  Rg_* \Omega_{U_{Y'}}^{q+c}[q+c]) 
  \end{equation}
  is \emph{injective}.
  This can be checked Zariski-locally at a point \(x \in X \subseteq Y\), so we may assume \(X \subseteq
  Y\) is a global complete intersection, say of \(t_1, \dots, t_c \in \sO_Y\).
  In that case the \(t_i \) define a Koszul resolution \(\cK^\bullet(t_i) \to
  \sO_X  \), and \emph{because} \(X' = Y' \times_Y X = V(t_1 \circ g, \cdots
  t_c\circ g)\) is smooth of codimension \(c\) by hypotheses, it must be that
  the \(t_i \circ g\) are also a regular sequence, hence 
  \[ L^i g^* \sO_X =  h^{-i} g^* \cK^\bullet(t_i) = \begin{cases}
    \sO_{X'}, & i =0 \\
    0 & \text{otherwise}
  \end{cases} \]
  in other words \(Lg^* \sO_X = \sO_{X'}\). Now using the fact that \(
  \Omega_X^q(\log \Delta_X)\) is locally free on \(X'\) we conclude 
  \[ Lg^* \imath_* \Omega_X^q(\log \Delta_X) [q] = g^* \imath_*
  \Omega_X^q(\log \Delta_X) [q] = \imath'_* g'^* \Omega_X^q(\log \Delta_X) [q]
  \]
  Next, applying derived adjunction to both sides of \eqref{eq:lazy-injection}
  gives a commutative diagram
  \begin{equation}
    \label{eq:der-adj-bothsides}
    {\footnotesize
    \begin{tikzcd}[column sep=small]
      R\sHom_Y(\imath_* \Omega_X^q(\log \Delta_X) [q],  Rg_*
      \Omega_{Y'}^{q+c}(\log \Delta_{Y'})[q+c]) \arrow[d, equals] \arrow[r] & 
      R\sHom_{U_Y}(\imath_* \Omega_{U_X}^q [q],  Rg_* \Omega_{U_{Y'}}^{q+c}[q+c]) \arrow[d, equals]\\
      Rg_* R\sHom_{Y'} (Lg^* \imath_*   \Omega_X^q(\log \Delta_X) [q],  \Omega_{Y'}^{q+c}(\log \Delta_{Y'})[q+c])  \arrow[d, equals] \arrow[r] & Rg_* R\sHom_{U_{Y'}} (Lg^* \imath_* \Omega_{U_X}^q [q],  \Omega_{U_{Y'}}^{q+c}[q+c]) \arrow[d, equals] \\
      Rg_* R\sHom_{Y'} (\imath'_* g'^* \Omega_X^q(\log \Delta_X) [q],  \Omega_{Y'}^{q+c}(\log \Delta_{Y'})[q+c]) \arrow[r] & Rg_* R\sHom_{U_{Y'}} (\imath'_* g'^* \Omega_{U_X}^q [q],  \Omega_{U_{Y'}}^{q+c}[q+c])
    \end{tikzcd}}
  \end{equation}
  Getting even more Zariski-local we may assume \(\Omega_X^q(\log \Delta_X)\) is
  \emph{free}, say generated by \(dx_1, \dots, dx_{n}\) and in that case 
  \begin{equation}
    \label{eq:dir-sum-decomp-both-factors}
    \begin{split}
      &R\sHom_{Y'} (\imath'_* g'^* \Omega_X^q(\log \Delta_X) [q],
  \Omega_{Y'}^{q+c}(\log \Delta_{Y'})[q+c])  \\
  &= (\prod_i R\sHom_{Y'}  (\sO_{X'}
  dx_i[q], \sO_{Y'}[q+c] ) )\otimes \Omega_{Y'}^{q+c}(\log \Delta_{Y'})
    \end{split}
  \end{equation}
  and by Grothendieck's fundamental local isomorphism  \cite[\S 2.5]{MR1804902}
  \begin{equation}
    \label{eq:fund-loc-iso}
    R\sHom_{Y'}  (\sO_{X'} [q], \sO_{Y'}[q+c] ) )\simeq \sExt_{Y'}^c(\sO_{X'},
  \sO_{Y'}) \simeq \det (\sI_{X'}/\sI_{X'})^\vee
  \end{equation}
  (the last 2 as sheaves
  supported in degree 0). In particular, this is an \emph{invertible sheaf on
  \(X'\)}, and it follows that the left hand side of
  \eqref{eq:dir-sum-decomp-both-factors} is a locally free sheaf (supported in
  degree 0) on \(X'\). Recalling \(X'\) is smooth and so in particular reduced,
  and since \(U_{Y'} \cap X'\) is a dense open (this is part of the hypothesis
  that \(X' \to Y'\) is a pulling map) the natural map 
  \begin{equation}
    \begin{split}
      & h^0 R\sHom_{Y'} (\imath'_* g'^* \Omega_X^q(\log \Delta_X) [q],
  \Omega_{Y'}^{q+c}(\log \Delta_{Y'})[q+c]) \\
  &\to  h^0 R\sHom_{Y'} (\imath'_* g'^* \Omega_X^q(\log \Delta_X) [q],
  \Omega_{Y'}^{q+c}(\log \Delta_{Y'})[q+c])|_{U_{Y'}}\\
  &\simeq h^0 R\sHom_{U_{Y'}} (\imath'_* g'^* \Omega_X^q(\log \Delta_X)|_{U_{Y'}} [q],
  \Omega_{Y'}^{q+c}(\log \Delta_{Y'})|_{U_{Y'}}[q+c])
    \end{split}
  \end{equation}
  is injective, where on the third line we have applied localization for
  \(\sExt\). Now left-exactness of \(g_* \)  gives an injection
  \begin{equation}
    \label{eq:finally-injective}
    \begin{split}
      &h^0 Rg_* R\sHom_{Y'} (\imath'_* g'^* \Omega_X^q(\log \Delta_X) [q],
      \Omega_{Y'}^{q+c}(\log \Delta_{Y'})[q+c]) \\
      &\to  h^0 Rg_* R\sHom_{U_{Y'}} (\imath'_* g'^* \Omega_X^q(\log \Delta_X)|_{U_{Y'}} [q],
      \Omega_{Y'}^{q+c}(\log \Delta_{Y'})|_{U_{Y'}}[q+c])
    \end{split}
  \end{equation} 
  To complete the proof, we use \eqref{eq:der-adj-bothsides} to identify the
  \emph{map} \eqref{eq:finally-injective} with \eqref{eq:lazy-injection}.
\end{proof}

\begin{corollary}
  \cref{lem:base-change-formula} holds in case \cref{item:pr-or-climm2} of \cref{lem:pr-or-climm}.
\end{corollary}

\begin{proof}
  This follows by applying cohomology with supports to \eqref{eq:proj-form-climm-case}.
\end{proof}

This completes our proof of \cref{lem:base-change-formula}. 

\begin{corollary}[{projection formula, compare with \cite[Prop.
1.1.16]{MR2923726}}]
  \label{cor:proj-formula}
  Let \(f: X \to Y\) be a map of smooth schemes admitting two different
  enhancements to maps of smooth schemes with supports,
  \[ (X, \Delta_X, \Phi_X) \to (Y,\Delta_Y, f(\Phi_X) ) \text{  pushing  and  }
  (X, f^*(\Delta'_Y), f^{-1}(\Phi_Y)) \to (Y, \Delta'_Y , \Phi_Y) \text{  pulling}
  \]
  Assume in addition that \(\Delta_X + f^*(\Delta'_Y)\) and \(\Delta_Y +
  \Delta'_Y\) are (\emph{reduced})
  snc divisors. Then 
  \[ (X, \Delta_X + f^*(\Delta'_Y), \Phi_X\cap f^{-1}(\Phi_Y)) \to (Y,\Delta_Y +
  \Delta'_Y, f(\Phi_X) \cap \Phi_Y )  \] is also a pushing map, and 
  \[ f_* (\beta \smile f^* \alpha) =  f_* \beta \smile \alpha  \in H^*
  (Y,\Delta_Y + \Delta'_Y, f(\Phi_X) \cap \Phi_Y )  \] for any \( \alpha \in H^*
  (Y, \Delta'_Y , \Phi_Y) \) and \(\beta \in  (X, \Delta_X, \Phi_X)\), where
  \(\smile\) is the cup product on log Hodge cohomology defined along the lines
  of \cite[\S 1.1.4, 2.4]{MR2923726}
\end{corollary}

\begin{proof}
  This is a formal consequence of \cref{lem:base-change-formula} and can be
  derived following the proof of \cite[Prop. 1.1.16]{MR2923726}. Again we use a
  factorization through the graph like
  \begin{equation}
    \small
    \begin{tikzcd}[column sep=small]
      (X, \Delta_X + f^*(\Delta'_Y), \Phi_X\cap f^{-1}(\Phi_Y)) \arrow[d, "\mathrm{id}_X \times \mathrm{id}_X"] \arrow[r, "f"] & (Y,\Delta_Y +    \Delta'_Y, f(\Phi_X) \cap \Phi_Y )  \arrow[dd, "\mathrm{id}_Y \times \mathrm{id}_Y"] \\
      (X \times X, \pr_1^*\Delta_X + \pr_2^* f^*(\Delta'_Y), \Phi_X \times f^{-1}(\Phi_Y)) \arrow[d, "\mathrm{id}_X \times f"] & \\
      (X\times Y, \pr_1^*\Delta_X +\pr_2^*\Delta'_Y,  \Phi_X \times \Phi_Y) \arrow[r, "f \times \mathrm{id}_Y"] & (Y \times Y, \pr_1^*\Delta_Y +\pr_2^*\Delta'_Y, f(\Phi_X) \times \Phi_Y)
    \end{tikzcd}
  \end{equation}
  Here \(f \times \mathrm{id}_Y\) on the bottom is a pushing morphism (since
  \(f|_{\Phi_X}\) is proper and \(f^* \Delta_Y = \Delta_X\)) and the right
  vertical map \(\mathrm{id}_Y \times \mathrm{id}_Y\) is a closed immersion
  transverse to \(f \times \mathrm{id}_Y\) since the outer rectangle is
  cartesian and \(X\) is smooth of the correct codimension. This means we are in
  a situation to apply \cref{lem:base-change-formula}, and that lemma plus the
  definition of cup products in terms of pullbacks along diagonals gives the
  desired identity.
\end{proof}


\section{Correspondences}
\label{sec:correspondences}
Given snc pairs with familes of supports \(\lsps{X}\) and \(\lsps{Y}\) with
dimensions \(d_X\) and \(d_Y\), as in
\cite[\S 1.3]{MR2923726} we may define a family of supports \(P(\Phi_X,
\Phi_Y)\) on \(X \times Y\) by 
\[ \begin{split}
  P(\Phi_X, \Phi_Y) := \{\text{closed subsets } Z \subseteq X \times Y  \, |& \,
\mathrm{pr}_Y|_Z \text{ is proper and for all } W\in \Phi_X, \\
&\mathrm{pr}_Y(\mathrm{pr}_X^{-1}(W) \cap Z) \in \Phi_Y \} \end{split} \] (the
conditions of \cref{def:2} are straightforward to verify). For convenience we
will let \(\Delta_{X\times Y} := \mathrm{pr}_X^* \Delta_X + \mathrm{pr}_Y^*
\Delta_Y\).
\begin{theorem}
  \label{thm:log-hodge-corresp}
  A class \(\gamma \in H^j_{P(\Phi_X, \Phi_Y)}(X \times Y, \ologd{X\times
  Y}{i}(-\mathrm{pr}_X^* \Delta_X))\) defines homomorphisms  
  \[    \cor (\gamma) : H_{\Phi_X}^q(X, \ologd{X}{p}) \to
  H_{\Phi_Y}^{q+j-d_X}(Y, \ologd{Y}{p+i-d_X}) \] by the formula
  \(\cor(\gamma)(\alpha) := \mathrm{pr}_{Y*}(\mathrm{pr}_X^* (\alpha) \smile
  \gamma)\). Moreover if \(\lsps{Z}\) is another snc pair with supports and
  \(\delta \in H^{j'}_{P(\Phi_Y, \Phi_Z)}(Y \times Z, \ologd{Y \times
  Z}{i'}(-\mathrm{pr}_Y^* \Delta_Y))\), then \[\mathrm{pr}_{X \times Z
  *}(\mathrm{pr}_{X \times Y}^* (\gamma) \smile \mathrm{pr}_{Y \times Z}^*
  (\delta)) \in H_{P(\Phi_X,\Phi_Z)}^{j  + j' - d_Y}(X \times Z, \ologd{X \times
  Z}{i + i' - d_Y}(-\mathrm{pr}_X^* \Delta_X)) \text{  and}\]
  \[\cor(\mathrm{pr}_{X \times Z *}(\mathrm{pr}_{X \times Y}^* (\gamma) \smile
  \mathrm{pr}_{Y \times Z}^* (\delta)) ) = \cor(\delta)\circ\cor(\gamma)  \] as
  homomorphisms \(H_{\Phi_X}^q(X, \ologd{X}{p}) \to H_{\Phi_Z}^{q+ j + j' - d_X
  - d_Y}(Z, \ologd{Z}{p+i+i' - d_X - d_Y})\). 
\end{theorem}

\begin{remark}
  The sheaves \(\ologd{X\times Y}{i}(-\mathrm{pr}_X^* \Delta_X)\) are particular
  instances of the sheaves \(\Omega_X^i(A, B)\) appearing in \cite[\S 4.2]{MR894379}.
\end{remark}

Such correspondences involving both log poles and ``log zeroes'' appear to have
been considered before at least in  crystalline cohomology, for example in work
of Mieda \cite{miedaIntegralLogCrystalline,miedaCycleClasses}. However, I was
unable to find any published proof of \cref{thm:log-hodge-corresp} in the
literature. 

\begin{proof}
  We make two observations: first, using \cref{lem:loc-descr-diff-log-poles} there are natural wedge product
  pairings
  \[\ologd{X\times Y}{p} \otimes \ologd{X \times Y}{i}(-\mathrm{pr}_X^*
  \Delta_X) \xrightarrow{\wedge} \Omega_{X \times Y}^{p+i}(\log \Delta_Y)\]
  Second, essentially by the definition of \(P(\Phi_X, \Phi_Y)\) the K\"unneth
  morphism on cohomology for the tensor product \(\ologd{X\times Y}{p} \otimes
  \ologd{X \times Y}{i}(-\mathrm{pr}_X^* \Delta_X)\) can be enhanced with
  supports as
  \[ \begin{split} H_{\mathrm{pr}^{-1}_X(\Phi_X)}^q(X \times Y, \ologd{X\times
    Y}{p}) & \otimes H^j_{P(\Phi_X, \Phi_Y)}(X\times Y , \ologd{X \times
    Y}{i}(-\mathrm{pr}_X^* \Delta_X)) \\
    &\to H_{\Psi}^{p+j}(X \times Y, \ologd{X\times Y}{p} \otimes \ologd{X \times
  Y}{i}(-\mathrm{pr}_X^* \Delta_X)) \end{split} \] where \(\Psi :=
  \pr_{Y*}^{-1}(\Phi_Z)\) (see \cite[\S 1.3.7, Prop. 1.3.10]{MR2923726}). Combining these 2 observations
  gives a pairing
  \[
  \begin{split}
    H_{\mathrm{pr}^{-1}_X(\Phi_X)}^q(X \times Y, \ologd{X\times Y}{p}) & \otimes
    H^j_{P(\Phi_X, \Phi_Y)}(X\times Y , \ologd{X \times Y}{i}(-\mathrm{pr}_X^* \Delta_X)) \\
    & \xrightarrow{\smile} H_{\Psi}^{p+j}(X \times Y, \Omega_{X \times
    Y}^{p+i}(\log \Delta_Y)) 
  \end{split}
  \] Now note that \(\mathrm{pr}_X: (X
  \times Y,  \Delta_{X \times Y}, \mathrm{pr}^{-1}_X(\Phi_X)) \to \lsps{X}\) is
  a pulling morphism, so by \cref{prop:pullb-pushf-morph} there is an induced map
  \(\mathrm{pr}_X^*: H_{\Phi_X}^q(X, \ologd{X}{p}) \to
  H_{\mathrm{pr}^{-1}_X(\Phi_X)}^q(X \times Y, \ologd{X\times Y}{p})\). On the
  other hand since \(\mathrm{pr}_Y: (X \times Y,\Delta_Y, \Psi ) \to \lsps{Y}\)
  is a pushing morphism, \cref{lem:pushforward} provides a  morphism \(\mathrm{pr}_{Y*}:
  H_{\Psi}^{p+j}(X \times Y, \Omega_{X \times Y}^{p+i}(\log \Delta_Y)) \to
  H_{\Phi_Y}^{q+j-d_X}(Y, \ologd{Y}{p+i-d_X}) \). Composing, we obtain the
  desired homomorphism 
  \[ \begin{split} H_{\Phi_X}^q(X, \ologd{X}{p}) &\xrightarrow{\mathrm{pr}_X^*}
    H_{\mathrm{pr}^{-1}_X(\Phi_X)}^q(X \times Y, \ologd{X\times Y}{p})\\
    &\xrightarrow{\smile \gamma} H_{\Psi}^{p+j}(X \times Y, \Omega_{X \times
    Y}^{p+i}(\log \Delta_Y)) \\
    & \xrightarrow{\mathrm{pr}_{Y*}} H_{\Phi_Y}^{q+j-d_X}(Y, \ologd{Y}{p+i-d_X})
  \end{split} \] 
  
  For the ``moreover'' half of the lemma, we again begin with a certain wedge
  product pairing, this time on \(X \times Y \times Z\):
  \begin{equation}
    \label{eq:wedge-on-xyz}
    \begin{split} 
      &\Omega_{X \times Y \times Z}^i(\log \mathrm{pr}_{X \times Y}^* \Delta_{X\times Y})(-\mathrm{pr}_X^* \Delta_X) \otimes \Omega_{X \times Y \times Z}^{i'}(\log \mathrm{pr}_{Y \times Z}^* \Delta_{Y\times Z})(-\mathrm{pr}_Y^* \Delta_Y) \\
      &\xrightarrow{\wedge} \Omega_{X \times Y \times Z}^{i + i'}(\log
      \mathrm{pr}_{X \times Z}^* \Delta_{X \times Z})(-\mathrm{pr}_X^* \Delta_X)
    \end{split}
  \end{equation}
  If \(V \in P(\Phi_X, \Phi_Y), W \in P(\Phi_Y, \Phi_Z)\) then unravelling
  definitions (again we refer to  \cite[\S 1.3.7, Prop. 1.3.10]{MR2923726} for a
  similar claim) we find:
  \begin{itemize}
    \item \(\mathrm{pr}_{X \times Z}|_{\mathrm{pr}_{X \times Y}^{-1}(V) \cap
    \mathrm{pr}_{Y \times Z}^{-1}(W)}\) is proper and
    \item \(\mathrm{pr}_{X \times Z}(\mathrm{pr}_{X \times Y}^{-1}(V) \cap
    \mathrm{pr}_{Y \times Z}^{-1}(W)) \in P(\Phi_X, \Phi_Z)\)
  \end{itemize}
  so that the K\"unneth morphism on cohomology associated to the left hand side
  of \eqref{eq:wedge-on-xyz} can be enhanced with supports like
  \[ 
  \begin{split}
    &H_{\mathrm{pr}_{X \times Y}^{-1}(P(\Phi_X, \Phi_Y))}^j (X \times Y \times Z, \Omega_{X \times Y \times Z}^i(\log
    \mathrm{pr}_{X \times Y}^* \Delta_{X\times Y})(-\mathrm{pr}_X^* \Delta_X))\\
    &\otimes  H_{\mathrm{pr}_{Y \times Z}^{-1}(P(\Phi_Y, \Phi_Z))}^{j'}(X \times Y \times Z, \Omega_{X \times Y \times
    Z}^{i'}(\log \mathrm{pr}_{Y \times Z}^* \Delta_{Y\times
    Z})(-\mathrm{pr}_Y^* \Delta_Y)) \\
    &\to H_\Sigma^{j+j'}(X \times Y \times Z, \Omega_{X \times Y \times Z}^i(\log \mathrm{pr}_{X \times Y}^* \Delta_{X\times Y})(-\mathrm{pr}_X^* \Delta_X) \otimes \Omega_{X \times Y \times Z}^{i'}(\log \mathrm{pr}_{Y \times Z}^* \Delta_{Y\times Z})(-\mathrm{pr}_Y^* \Delta_Y) )
  \end{split}
  \]
  where \(\Sigma :=  \pr_{X \times Z *}^{-1}( P(\Phi_X, \Phi_Z))\). 
  
  Since \(\mathrm{pr}_{X \times Y} : (X \times Y \times Z, \mathrm{pr}_{X \times
  Y}^* \Delta_{X\times Y}, \mathrm{pr}_{X \times Y}^{-1}(P(\Phi_X, \Phi_Y))) \to
  (X \times Y,  \Delta_{X\times Y}, P(\Phi_X, \Phi_Y))\) is a pulling morphism,
  \cref{prop:pullb-pushf-morph} gives an induced morphism 
  \[\ologd{X \times Y}{i} \to Rf_* \Omega_{X \times Y \times Z}^i(\log
  \mathrm{pr}_{X \times Y}^* \Delta_{X\times Y});\] twisting by \(-
  \Delta_{X\times Y}\) and applying the projection formula gives a morphism
  \[\ologd{X \times Y}{i}(- \Delta_{X\times Y}) \to Rf_* \big{(} \Omega_{X
  \times Y \times Z}^i(\log \mathrm{pr}_{X \times Y}^* \Delta_{X\times
  Y})(-\mathrm{pr}_{X \times Y}^* \Delta_{X\times Y}) \big{)} \] and then taking
  cohomology with supports along \(P(\Phi_X, \Phi_Y)\) and using
  \cref{prp:coho-with-supp} gives a modified pullback map 
  \begin{equation}
    \begin{split}
      &H_{P(\Phi_X, \Phi_Y)}^j(X \times Y, \ologd{X \times Y}{i}(- \Delta_{X\times    Y})) \\
      &\to H_{\mathrm{pr}_{X \times Y}^{-1}(P(\Phi_X, \Phi_Y))}^j (X \times Y
    \times Z, \Omega_{X \times Y \times Z}^i(\log \mathrm{pr}_{X \times Y}^*
    \Delta_{X\times Y})(-\mathrm{pr}_X^* \Delta_X))
    \end{split}
  \end{equation}
  and a similar argument gives
  a modified pullback 
  \begin{equation}
    \begin{split}
      &H_{P(\Phi_Y, \Phi_Z)}^{j'}(Y \times Z, \ologd{Y \times Z}{i'}(- \Delta_{Y
  \times Z})) \\
  & \to H_{\mathrm{pr}_{Y \times Z}^{-1}(P(\Phi_Y, \Phi_Z))}^{j'} (X
  \times Y \times Z, \Omega_{X \times Y \times Z}^{i'}(\log \mathrm{pr}_{Y
  \times Z}^* \Delta_{Y \times Z})(-\mathrm{pr}_X^* \Delta_Y))
    \end{split}
  \end{equation}
  On the other
  hand, \(\mathrm{pr}_{X\times Z} : (X \times Y \times Z, \mathrm{pr}_{X \times
  Z}^* \Delta_{X\times Y}, \Sigma) \to (X \times Z,  \Delta_{X \times Z},
  P(\Phi_X, \Phi_Z))\) is a pushing morphism and hence by \cref{lem:pushforward}
  induces morphisms
  \[ R \mathrm{pr}_{X\times Z *} R \underline{\Gamma}_{\Sigma}(\Omega_{X \times
  Y \times Z}^{\dim X \times Y \times Z- k}(\log \mathrm{pr}_{X \times Z}^*
  \Delta_{X\times Y})) \to  R \underline{\Gamma}_{P(\Phi_X, \Phi_Z)} \Omega_{X
  \times Z}^{\dim X \times Z - k}(\log \Delta_{X \times Z})[- \dim Z]  \] for
  all \( k\); twisting by \(-\mathrm{pr}_X^* \Delta_X\) and applying the
  projection formula this becomes
  \begin{equation}
    \begin{split}
      &R \mathrm{pr}_{X\times Z *} R \underline{\Gamma}_{\Sigma}(\Omega_{X \times
  Y \times Z}^{\dim X \times Y \times Z- k}(\log \mathrm{pr}_{X \times Z}^*
  \Delta_{X\times Y})(-\mathrm{pr}_X^* \Delta_X)) \\
  &\to  R
  \underline{\Gamma}_{P(\Phi_X, \Phi_Z)} \Omega_{X \times Z}^{\dim X \times Z -
  k}(\log \Delta_{X \times Z})(-\mathrm{pr}_X^* \Delta_X)[- \dim Z]
    \end{split}
  \end{equation}
  Now
  letting \(k = \dim X \times Y \times Z - i - i'\), the induced morphisms of
  cohomology with supports are
  \begin{equation}
    \begin{split}
      &H_\Sigma^{j+ j'}(X \times Y \times Z, \Omega_{X \times Y \times Z}^{i +
    i'}(\log \mathrm{pr}_{X \times Z}^* \Delta_{X\times Y})(- \mathrm{pr}_X^*
    \Delta_X))\\
    &\to H_{P(\Phi_X, \Phi_Z)}^{j+j' - \dim Z}(X \times Z, \Omega_{X
    \times Z}^{i+i'-\dim Z}(\log \Delta_{X \times Z})(- \mathrm{pr}_X^*
    \Delta_X))
    \end{split}
  \end{equation}
  Combining the above ingredients, we obtain a bilinear pairing 
  \[ 
  \begin{split}
    &H_{P(\Phi_X, \Phi_Y)}^j(X \times Y, \ologd{X \times Y}{i}(- \Delta_{X\times
    Y}))\otimes H_{P(\Phi_Y, \Phi_Z)}^{j'}(Y \times Z, \ologd{Y \times Z}{i'}(- \Delta_{Y    \times Z})) \\ &\to H_{P(\Phi_X,
    \Phi_Z)}^{j+j' - \dim Z}(X \times Z, \Omega_{X \times Z}^{i+i'-\dim Z}(\log \Delta_{X \times Z})(- \mathrm{pr}_X^* \Delta_X))
  \end{split}  
  \] sending \( \gamma \otimes \delta \mapsto \mathrm{pr}_{X \times Z
  *}(\mathrm{pr}_{X \times Y}^* (\gamma) \smile \mathrm{pr}_{Y \times Z}^*
  (\delta))\). It remains to be seen that 
  \[  \cor(\mathrm{pr}_{X \times Z *}(\mathrm{pr}_{X \times Y}^* (\gamma) \smile
  \mathrm{pr}_{Y \times Z}^* (\delta)) ) = \cor(\delta) \circ \cor(\gamma) \]
  and for this we will make repeated use of \cref{lem:base-change-formula}. Consider
  the diagram of smooth schemes
  
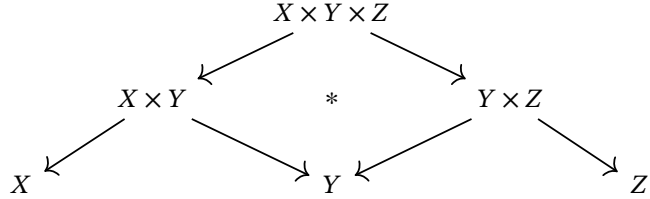
\[
    \begin{tikzcd}
      && X \times Y \times Z \arrow[ld] \arrow[rd] \arrow[dd, phantom, "\ast"] && \\
      & X \times Y  \arrow[ld] \arrow[rd] && Y \times Z \arrow[ld] \arrow[rd] &\\
      X && Y && Z
    \end{tikzcd}  
  \]
  where all morphisms are projections. There are various ways to enhance this to
  include supports; here we add the family of supports \(\Psi\) on \(X \times
  Y\) defined above. Then in the cartesian diagram \((\ast)\),
  \(\mathrm{pr}_{Y}: (X \times Y, \Psi) \to (Y, \Phi_Y)\) and \( \mathrm{pr}_{Y
  \times Z}: (X \times Y \times Z , \mathrm{pr}_{X \times Y}^{-1}\Psi) \to (Y
  \times Z, \mathrm{pr}_Y^{-1}\Phi_Y) \) are pushing morphisms, whereas
  \(\mathrm{pr}_{X \times Y}\) and \(\mathrm{pr}_Y\) are pulling morphisms. At
  the same time, we have a pulling morphism \(\mathrm{pr}_{X \times Z}: (X
  \times Y \times Z , \mathrm{pr}_{X \times Z}^{-1}(P(\Phi_Y, \Phi_Z))) \to (Y
  \times Z, P(\Phi_Y, \Phi_Z)) \). To be precise in what follows, whenever
  ambiguity is possible we will use
  notation like \(\mathrm{pr}^{X \times Y}_X\) to denote the projection \(X
  \times Y \to X\), \(\mathrm{pr}^{X \times Y \times Z}_X\) to denote the
  projection \(X \times Y \times Z \to X\) and so on.
  
  Applying \cref{cor:proj-formula} first to \(\mathrm{pr}_{X \times Z}\) we see
  that 
  \[ \mathrm{pr}_{Y\times Z *}(\mathrm{pr}_{X \times Y}^*(\mathrm{pr}_X^{X\times
  Y*} \alpha \smile \gamma) \smile \mathrm{pr}_{Y \times Z}^* \delta) =
  \mathrm{pr}_{Y\times Z*}(\mathrm{pr}_{X \times Y}^*(\mathrm{pr}_X^{X\times Y*}
  \alpha \smile \gamma)) \smile \delta \] and then applying \cref{lem:base-change-formula} to \((\ast)\) shows \[\mathrm{pr}_{Y\times Z*}(\mathrm{pr}_{X \times
  Y}^*(\mathrm{pr}_X^{X\times Y*} \alpha \smile \gamma)) =
  \mathrm{pr}_Y^{Y\times Z*} (\mathrm{pr}^{X \times
  Y}_{Y*}(\mathrm{pr}_X^{X\times Y*} \alpha \smile\gamma)) =
  \mathrm{pr}_Y^{Y\times Z*} \cor(\gamma)(\alpha)\] so that
  \[\mathrm{pr}_{Y\times Z *}(\mathrm{pr}_{X \times Y}^*(\mathrm{pr}_X^{X\times
  Y*} \alpha \smile \gamma) \smile \mathrm{pr}_{Y \times Z}^* \delta)
  =\mathrm{pr}_Y^{Y\times Z*} \cor(\gamma)(\alpha)\smile \delta  \] Applying
  \(\mathrm{pr}^{Y\times Z}_{Z *}\) we conclude that 
  \begin{equation}
    \label{eq:comp-of-corresp}
    \cor \delta  (\cor
    \gamma)(\alpha)) = \mathrm{pr}^{X\times Y \times Z}_{Z*}(\mathrm{pr}_X^{X\times  Y \times Z*} \alpha \smile \mathrm{pr}_{X \times Y}^* \gamma \smile   \mathrm{pr}_{Y \times Z}^* \delta) 
  \end{equation}
  Finally, we rewrite the right hand side
  as 
  \[ \mathrm{pr}^{X\times Z}_{Z*} \mathrm{pr}_{X \times Z*}(\mathrm{pr}_{X\times
  Z}^*\mathrm{pr}^{X\times Z *}_X \alpha \smile \mathrm{pr}_{X \times Y}^*
  \gamma \smile \mathrm{pr}_{Y \times Z}^* \delta) \] and apply
  \cref{cor:proj-formula} to \(\mathrm{pr}_{X \times Z}\) (with the pushing
  morphism \((X \times Y \times Z , \Sigma) \to (X \times Z,  P(\Phi_X,
  \Phi_Z))\) and pulling morphism \((X \times Y \times Z , \mathrm{pr}^{X\times
  Y \times Z -1}_{X}(\Phi_X)) \to (X \times Z,  \mathrm{pr}^{X \times
  Z-1}_{X}(\Phi_X))\)) to arrive at \[  \mathrm{pr}_{X \times
  Z*}(\mathrm{pr}_{X\times Z}^*\mathrm{pr}^{X\times Z *}_X \alpha \smile
  \mathrm{pr}_{X \times Y}^* \gamma \smile \mathrm{pr}_{Y \times Z}^* \delta) =
  \mathrm{pr}^{X\times Z *}_X \alpha \smile\mathrm{pr}_{X \times Z*}(
  \mathrm{pr}_{X \times Y}^* \gamma \smile \mathrm{pr}_{Y \times Z}^* \delta) \]
  Applying \(\mathrm{pr}^{X\times Z}_{Z*} \) on both sides shows that the right
  hand side of \eqref{eq:comp-of-corresp} is \(\cor(\mathrm{pr}_{X \times Z*}(
  \mathrm{pr}_{X \times Y}^* \gamma \smile \mathrm{pr}_{Y \times Z}^*
  \delta)(\alpha)\), as desired.
\end{proof}

\begin{remark}
  There is a Grothendieck-Serre dual approach to such correspondences, where
classes  \(\gamma \in H^j_{P(\Phi_X, \Phi_Y)}(X \times Y, \ologd{X\times
Y}{i}(-\mathrm{pr}_Y^* \Delta_Y))\) define homomorphisms  \[H^q(X,
\ologd{X}{p}(- \Delta_X)) \to H^{q+j-d_X}(Y,
\ologd{Y}{p+i-d_X}(- \Delta_Y)).\] The construction is formally similar.
\end{remark}

\printbibliography

\appendix

\section{Attempts to construct a fundamental class of a thrifty birational equivalence}
\label{sec:attempts}

As mentioned in \Cref{sec:intro} inspiration for this work was the following
remarkable theorem of Chatzistamatiou-R\"ulling:
\begin{theorem}[{\cite[Thm. 3.2.8]{MR2923726} (see also \cite[Thm. 1.1]{MR3427575}, \cite[Thm. 1.6]{kovacsRationalSingularities2020})}]
    \label{thm:cr-hdi-struct}
    Let \(k\) be a perfect field and let \(S\) be a scheme. Suppose \(X\) and
    \(Y\) are two separated, finite type \(k\)-schemes which are
    \begin{enumerate}
    \item smooth over \(k\) and 
    \item \textbf{properly birational} over \(S\) in the sense that
      there is a commutative diagram
      \begin{equation}
        \label{eq:1}
        \begin{tikzcd}
          & Z \arrow[dl, "r"'] \arrow[dr, "s"]  & \\
          X \arrow[dr, "f"']  \arrow[rr, phantom, "\circlearrowleft"] &   & Y \arrow[dl, "g"] \\
          & S & \\
        \end{tikzcd}
      \end{equation}
      with \(r\) and \(s\) proper birational morphisms.  
    \end{enumerate}
    Let \(n = \dim X = \dim Y = \dim Z\). Then, there are isomorphisms of sheaves
    \begin{equation}
      \label{eq:cr-correspondence}
       R^{i} f_{*} \sO_X \isom  R^{i}
      g_{*} \sO_Y  \text{  and  } R^{i} f_{*} \omega_X \isom  R^{i}
      g_{*} \omega_Y \text{  for all  } i, 
    \end{equation}
\end{theorem}
This result implies, for example, that if \(S\) is a variety over a perfect
field \(k\) with a \textbf{rational resolution}, that is, a resolution of
singularities \(f: X \to S\) such that \(Rf_* \strshf{X} = \strshf{S}\), then
every other resolution \(g: Y \to S\) satisfies \(Rg_* \strshf{Y} = \strshf{S}\)
and is hence also rational. In characteristic 0 this was a corollary of
Hironaka's resolution of singularities \cite{MR0199184}; in positive
characteristic it remained open until 2011. 

The original proof in \cite[Thm. 3.2.8]{MR2923726} makes use of a cycle morphism
\(\mathrm{cl}: CH^*(X) \to H^*(X, \Omega^*_X)\) from Chow cohomology to Hodge
cohomology, which is ultimately applied to a cycle \(Z\subset X \times Y\)
obtained from a proper birational equivalence. That cycle morphism satisfies 2
essential properties: the first is that it is compatible with \emph{correspondences}:
here Chow correspondences are homomorphisms \[ CH^*(X) \to CH^*(Y) \text{  of
the form  } \alpha \mapsto \mathrm{pr}_{Y*} (\mathrm{pr}_X^* \alpha \smile
\gamma) \text{  for some  } \gamma \in CH^*(X \times Y)  \] where \(\smile\) is
the cup product induced by intersecting cycles; Hodge correspondences are
defined in a similar way. The second key property is a compatibility with the
filtrations \[CH^n(X \times Y) = F^0 CH^n(X \times Y)  \supseteq F^1 CH^n(X
\times Y) \supseteq \cdots \supseteq F^{\dim Y} CH^n(X \times Y) \supseteq 0\]
where \(F^c CH^n(X \times Y) \) is the subgroup generated by cycles \(Z
\subseteq X \times Y\) such that \( \codim(\mathrm{pr}_{Y} Z \subseteq Y) \geq
c\), and 
\[H^n(X \times Y, \Omega_{X\times Y}^m) = F^0 H^n(X \times Y, \Omega_{X\times
Y}^m)  \supseteq F^1 CH^*(X \times Y) \supseteq \cdots \supseteq F^{\dim Y}
H^n(X \times Y, \Omega_{X\times Y}^m) \supseteq 0\] where \(F^c H^n(X \times Y,
\Omega_{X\times Y}^m) \) is the image of the map \(H^n(X \times Y, \oplus_{j =
c}^m \Omega_X^{m -j} \boxtimes \Omega_Y^j) \to H^n(X \times Y, \Omega_{X\times
Y}^m) \) coming from the K\"unneth decomposition.

It is natural to ask if a similar method can be applied to prove an analogue of
\cref{thm:cr-hdi-struct} for \emph{pairs}, which might read something like
\cref{conj:hdi-log-struct} below. In order to state this analogue, we require a
few additional definitions.
For the remainder of
this appendix we work over a fixed perfect field \(k\).
\begin{definition}[{slightly simplified version of \cite[Def.
  1.5]{MR3057950}}]
  \label{def:pair}
  A \textbf{pair \(\lsp{X}\)} over \(k\) will mean
  \begin{itemize}
    \item a reduced, equidimensional and \(S_2\) scheme \(X\) of finite type
    over \(k\) admitting a dualizing complex , together with
    \item a \(\QQ\)-Weil divisor \(\Delta_{X} = \sum_i a_i D_i\) on \(X\) such
    that no irreducible component \(D_i \) of \(\Delta_X\) is contained in
    \(\Sing(X)\).
  \end{itemize}   
\end{definition}
\begin{definition}
  \label{def:abstract-strata-intro}
  A \emph{stratum} of a simple normal crossing pair \((X, \Delta_X = \sum_i
  D_i)\) is a connected (equivalently, irreducible) component of an intersection
  \(D_J = \cap_{j \in J}D_j\).
\end{definition}
Given any pair \(\lsp{X}\), there is a largest open set \(U \subseteq X \) such
that  \((U, \Delta_X |_U)\) is a simple normal crossing pair, and we will refer
to the resulting simple normal crossing pair as \(\snc \lsp{X} := (U, \Delta_X
|_U)\).
\begin{definition}[
  {compare with \cite[Def. 2.79-2.80]{MR3057950}, \cite[\S 1,
  discussion before Def. 10]{MR3539921}}
  ]
  \label{def:thriftiness}
  Let \((S, \Delta_S = \sum_i D_i)\) be a pair,
  and assume 
  \(\Delta_S\) is reduced and effective. A separated, finite type birational
  morphism \(f: X \to S\) is \emph{thrifty with respect to} \(\Delta_S \) if
  and only if 
  \begin{enumerate}
    \item \label{item:over-gen-pts-intro} \(f\) is an isomorphism over the
    generic point of every stratum of \(\snc(S, \Delta_S)\) and
    \item \label{item:at-gen-pts-intro} letting \(\tilde{D}_i = f^{-1}_*D_i\)
    for \(i = 1, \dots, N\) be the strict transforms of the divisors \(D_i\),
    and setting \(\Delta_X := \sum_i \tilde{D}_i\), the map \(f\) is an
    isomorphism at the generic point of every stratum of \(\snc(X,
    \Delta_X)\).
  \end{enumerate}
\end{definition}
\begin{conjecture}
    \label{conj:hdi-log-struct}
    Let  \(k\) be a perfect field, let \(S\) be a scheme and let \(\lsp{X}\) and
    \(\lsp{Y}\) be simple normal crossing pairs over \(k \). Suppose \(\lsp{X}
    \) and \( \lsp{Y}\) are properly birational over \(S\) in the sense that
    there is a commutative diagram
    \begin{equation}
      \label{eq:4}
      \begin{tikzcd}
        & \lsp{Z} \arrow[dl, "r"'] \arrow[dr, "s"]  & \\
        \lsp{X} \arrow[dr, "f"']  \arrow[rr, phantom, "\circlearrowleft"] &   & \lsp{Y} \arrow[dl, "g"] \\
        & S & \\      
      \end{tikzcd}
    \end{equation}
    where \(r\), \(s\) are proper and birational morphisms, and assume \(
    \Delta_{Z} = r_{*}^{-1}\Delta_{X} = s_{*}^{-1} \Delta_{Y} \).
    If \(r\) and \(s\) are \emph{thrifty}, then there are quasi-isomorphisms
    \begin{equation}
      \label{eq:5}
      Rf_* \lstrshf{X} \simeq Rg_* \lstrshf{Y} \text{  and  } Rf_* \lcanshf{X} \simeq Rg_* \lcanshf{Y}.
    \end{equation}
\end{conjecture}
Following \cite{MR2923726} closely, one might begin by replacing the ordinary
sheaves of differentials \(\Omega_X\) appearing in Hodge cohomology with sheaves
of differentials with log poles \(\Omega_X(\log \Delta_X)\) and attempt to
implement a similar strategy, i.e. starting a cycle \(Z \subset X \times Y \)
representing a thrifty proper birational equivalince, producing a correspondence
in logarithmic Hodge cohomology and analyzing its properties. 

Ultimately even the correspondences of \Cref{sec:correspondences} seem to be
insufficient to deal with thrifty proper birational equivalences, as we
illustrate in \Cref{sec:obstructions} below. The problem we encounter is elementary:
looking at the recipe for the Hodge class \(\cl(Z)\) of a subvariety \(Z \subseteq
X\), where \(Z\) and \(X\) are smooth an projective (outlined in \cite[Ex.
III.7.4]{MR0463157}), we see that \(\cl(Z)\) ultimately comes from the trace
linear functional \(\tr : H^{\dim Z}(Z, \omega_Z) \to k \), or Serre-dually the
element \(1 \in H^0(Z, \sO_Z)\). Due to the introduction of log poles and zeroes
in \Cref{sec:correspondences}, trying to follow that recipe we pass through
cohomology groups of the form  \(H^{\dim Z}(Z, \omega_Z(D))\), or dually
\(H^0(Z, \sO_Z(-D))\) where \(D\) is an (often non-0 in cases of interest)
effective Cartier divisor on \(Z\), and so there simply is no ``\(1\)'' to be
had. 

Beyond the difficulties described in the previous paragraph, when attempting to
formulate a logarithmic variant of Chatzistamatiou-R\"ulling's cycle morphism
argument one is hampered by the fact that we are still in the early days of
logarithmic Chow theory 
. It is not clear to the author
which logarithmic variant of Fulton's \(CH^*\), if any, could be used to
construct a logarithmic cycle morphism with all of the desired properties.
Further investigation of this question could be an interesting topic of future
research.

Despite the aforementioned challenges, it is possible to prove a result almost
identical to \cref{conj:hdi-log-struct} by entirely different methods
\cite{GodfreyHDI}.\footnote{The reason the result is only ``almost identical''
is that in \cite{GodfreyHDI} we require ostensibly
stronger hypotheses on the base scheme \(S\) (namely that it is excellent and
noetherian), but it is possible that even in the situation of
\cref{thm:cr-hdi-struct,conj:hdi-log-struct} one can reduce to this case, for
example using noetherian approximation.}

\subsection{Obstructions to obtaining log Hodge correspondences from thrifty birational equivalences}
\label{sec:obstructions}

Let \(\lsp{X}, \lsp{Y}\) be simple normal crossing pairs, and assume in addition
that \(X, Y\) are connected and proper. Let \(Z \subseteq X \times Y \) be a
smooth closed subvariety with codimension \(c \). In
this situation the fundamental class of \(\mathrm{cl}(Z) \in H^c(X \times Y,
\Omega_{X \times Y}^c)\) (no log poles yet) can be described using only Serre
duality, as follows (we refer to \cite[Ex. III.7.4]{MR0463157}). the composition 
\begin{equation}
  \label{eq:res-then-trace}
  H^{\dim Z} (X \times Y, \Omega_{X \times Y}^{\dim Z}) \to H^{\dim Z} (Z,
  \Omega_{Z}^{\dim Z}) \xrightarrow{\tr} k
\end{equation} (where \(\tr\) is the trace map of
Serre duality) is an element of 
\begin{equation}
  \label{eq:serre-duality}
  H^{\dim Z} (X \times Y, \Omega_{X \times Y}^{\dim Z})^\vee \simeq H^c(X
  \times Y, \Omega_{X \times Y}^c)
\end{equation}
which we may \emph{define} to be \(\mathrm{cl}(Z)\).\footnote{It may then be
non-trivial to verify this agrees with other definitions, especially if we worry about signs, but we will not need that level of detail for what follows.}
In light of \cref{thm:log-hodge-corresp} we might hope to modify
\cref{eq:res-then-trace,eq:serre-duality} to obtain a class in \( H^c(X \times
Y, \Omega_{X \times Y}^c(\log \Delta_{X\times Y})(-\mathrm{pr}_X^* \Delta_X))
\). Let us focus on the case where 
\begin{itemize}
  \item \(\mathrm{pr}_X |_Z :Z \to X\), \(\mathrm{pr}_Y |_Z :Z \to Y\) are both
  thrifty and birational, so in particular \(c = \dim X = \dim Y =: d\) and
  \item \((\mathrm{pr}_X|_Z)_*^{-1} \Delta_X = (\mathrm{pr}_Y|_Z)_*^{-1}
  \Delta_Y =: \Delta_Z\)
\end{itemize}
To keep the notation under control, set \(\pi_X := \mathrm{pr}_X|_Z\) and
\(\pi_Y := \mathrm{pr}_Y|_Z\).

In this situation letting \(\iota: Z  \to X \times Y\) be the inclusion there is a natural map 
\[
  \begin{split}
    d \iota^\vee &: \Omega_{X \times Y}^{d}(\log \Delta_{X\times Y}) \to  \iota_* \Omega_{Z}^{d}(\log \Delta_{X\times Y}|_Z)  \text{ and twisting by \(-\mathrm{pr}_Y^*\Delta_Y\) gives a map} \\
    & \Omega_{X \times Y}^{d}(\log \Delta_{X\times Y})(-\mathrm{pr}_Y^*\Delta_Y) \to  \iota_* \Omega_{Z}^{d}(\log \Delta_{X\times Y}|_Z)(-\mathrm{pr}_Y^*\Delta_Y|_Z) = \iota_* \Omega_{Z}^{d}(\log \Delta_{X\times Y}|_Z)(-\pi_Y^* \Delta_Y)
  \end{split}  
\]
To identify \(\Omega_{Z}^{d}(\log \Delta_{X\times Y}|_Z)(-\mathrm{pr}_X^*
\Delta_X|_Z)\), write 
\[
\begin{split}
  (\pi_X)^* \Delta_X &= (\pi_X)^{-1}_* \Delta_X + E_X =
  \Delta_Z + E_X  \text{ and }\\
  (\pi_Y)^* \Delta_Y &= (\pi_Y)^{-1}_* \Delta_Y + E_Y =
  \Delta_Z + E_Y 
\end{split}  
\] 
so that \(\Delta_{X \times Y}|_Z = (\pi_X)^* \Delta_X + (\pi_Y)^* \Delta_Y  = 2
\Delta_Z + E_X + E_Y\). While the hypotheses guarantee \(\Delta_Z\) is reduced
it may be that \(E_X, E_Y\) are non-reduced --- however something can be said
about their multiplicities. If \(E_X = \sum_i a_X^i E_X^i, E_Y = \sum_i a_Y^i
E_Y^i\) where the \(E_X^i, E_Y^i\) are irreducible, then by a generalization of
\cite[Prop. 3.6]{MR0463157} (see also \cite[\S 2.10]{MR3057950}), \[ a_X^i = \mathrm{mlt}(\pi_X (E_X^i) \subseteq
\Delta_X)\] and since \(\Delta_X\) is a reduced effective simple normal crossing
divisor, if in addition we write \(\Delta_X = \sum_i D_X^i\), then
\(\mathrm{mlt}(\pi_X (E_X^i) \subseteq \Delta_X) = \lvert \{i \, | \, \pi_X
(E_X^i) \subseteq D_X^i\} \rvert \). The thriftiness hypothesis that \( \pi_X
(E_X^i) \) is not a stratum then implies \(a_X^i  = \mathrm{mlt}(\pi_X (E_X^i)
\subseteq \Delta_X) < \codim(\pi_X (E_X^i) \subset X)\). Since differentials
with log poles are insensitive to multiplicities, we have 
\[ \Omega_{Z}^{d}(\log \Delta_{X\times Y}|_Z) = \omega_Z(\Delta_Z +
E_X^{\textup{red}}+ E_Y^{\textup{red}}) \] where \(-^\textup{red}\) denotes the
associated reduced effective divisor. Then 
\[
\begin{split}
  \Omega_{Z}^{d}(\log \Delta_{X\times Y}|_Z)(-\pi_Y^* \Delta_Y) &=
  \omega_Z(\Delta_Z + E_X^{\textup{red}}+ E_Y^{\textup{red}} -\Delta_Z -E_Y ) \\
  &\omega_Z(E_X^{\textup{red}} + (E_Y^{\textup{red}} - E_Y)) = \omega_Z(\sum_i E_X^i + \sum_i(1-a_Y^i )E_Y^i)
\end{split}  
\]
The upshot is that we have an induced map 
\begin{equation}
  \label{eq:twisted-res}
  H^d(X \times Y, \Omega_{X \times Y}^{d}(\log \Delta_{X\times Y})(-\mathrm{pr}_Y^*\Delta_Y)) \to H^d(Z, \omega_Z(E_X^{\textup{red}} + (E_Y^{\textup{red}} - E_Y)))
\end{equation}
Here the left hand side is Serre dual to \( H^d(X \times Y, \Omega_{X \times
Y}^{d}(\log \Delta_{X\times Y})(-\mathrm{pr}_X^*\Delta_X)) \), so the
\(k\)-linear dual of \eqref{eq:twisted-res} is a morphism
\[ H^d(Z, \omega_Z(E_X^{\textup{red}} + (E_Y^{\textup{red}} - E_Y)))^\vee \to
H^d(X \times Y, \Omega_{X \times Y}^{d}(\log \Delta_{X\times
Y})(-\mathrm{pr}_X^*\Delta_X))\] Unfortunately\footnote{at least for the
purposes of constructing log Hodge cohomology classes of subvarieties ...}
\(H^d(Z, \omega_Z(E_X^{\textup{red}} + (E_Y^{\textup{red}} - E_Y)))\) is often
0. If \(E_X\) and \(E_Y\) are both reduced (an explicit example where this holds
will be given below), then \(H^d(Z, \omega_Z(E_X^{\textup{red}} +
(E_Y^{\textup{red}} - E_Y))) = H^d(Z, \omega_Z(E_X))\). If in addition \(E_X
\neq 0\), we obtain \(H^d(Z, \omega_Z(E_X)) = 0\) by an extremely weak (but
characteristic independent) sort of Kodaira vanishing:

\begin{lemma}
  Let \(Z\) be a proper variety over a field \(k\) with dimension \(d\), and
  assume \(Z\) is normal and Cohen-Macaulay. If \(D \subset Z\) is a non-0
  effective Cartier divisor on \(Z\) then \(H^d(Z, \omega_Z(D)) = 0\).
\end{lemma}

\begin{proof}
  By Serre duality \(H^d(Z, \omega_Z(D)) = H^0(Z, \mathscr{O}_Z(-D))\), which
  vanishes by the classic fact that ``a nontrivial line bundle and its inverse
  can't both have non-0 global sections.'' Since I am not aware of a specific
  reference, here is a proof:
  
  Suppose towards contraditction that there is a non-0 global section \(\sigma
  \in H^0(Z, \mathscr{O}_Z(-D)) \) --- then the composition
  \[
  \begin{tikzcd}
    \mathscr{O}_Z \arrow[r, "\sigma"] \arrow[rr, bend right, "\tau"] & \mathscr{O}_Z(-D) \arrow[r, hook] &  \mathscr{O}_Z
  \end{tikzcd}  
  \]
  is non-0. By \cite[\href{https://stacks.math.columbia.edu/tag/0358}{Tag
  0358}]{stacks-project} \(H^0(Z,\mathscr{O}_Z)\) is a (normal) domain, and
  since it's also a finite dimensional \(k\)-vector space it must be an
  extension field of \(k\). But then \(\tau \in H^0(Z,\mathscr{O}_Z)\) is
  invertible hence surjective, so \(\mathscr{O}_Z(-D) \inj \mathscr{O}_Z\) is
  surjective, which is a contradiction since by hypothesis the cokernel
  \(\mathscr{O}_D \neq 0\).
\end{proof}

\begin{example}
  Let \(X = \mathbb{P}^2\) and let \(\Delta_X \subset X\) be a line. Let \(p \in
  L \) be a \(k\)-point, let \(Y = \Bl_p X\) and let \(\Delta_Y = \tilde{L} =\)
  the strict transform of \(L\). Finally let \(f : Y \to X\) be the blowup map
  and let \(Z = (f \times \mathrm{id})(Y) \subset X \times Y\). In this case
  (with all notation as above) \(\pi_X \circ (f \times \mathrm{id}) = f\) and
  \(\pi_Y \circ (f \times \mathrm{id})= \mathrm{id}_Y\), so under the
  isomorphism \(f \times \mathrm{id} : Y \simeq Z\), \(E_X \) is the exceptional
  divisor of \(f\) (with multiplicity 1). On the other hand \(E_Y = 0\). In
  particular \(E_X\) and \( E_Y\) are reduced and \(E_X \neq 0\) so from the
  above discussion \(H^2(Z, \omega_Z(E_X)) = 0\).
\end{example}

\end{document}